\documentclass[12pt, reqno]{amsart}
\usepackage{amsmath, amsthm, amscd, amsfonts, amssymb, graphicx, color}
\usepackage{subcaption}
\usepackage[bookmarksnumbered, colorlinks, plainpages]{hyperref}
\usepackage{graphicx}
\usepackage{tikz}
\usepackage{mathrsfs}
\usepackage{tkz-euclide}
\usetikzlibrary{decorations.text,calc,arrows.meta}
\hypersetup{colorlinks=true,linkcolor=red, anchorcolor=green,
citecolor=cyan, urlcolor=red, filecolor=magenta, pdftoolbar=true}

\usepackage{tkz-berge}
\usetikzlibrary{3d}

\newcommand\pgfmathsinandcos[3]{%
  \pgfmathsetmacro#1{sin(#3)}%
  \pgfmathsetmacro#2{cos(#3)}%
}

\usepackage{pgfplots}
\newtheorem{theorem}{Theorem}[section]
\newtheorem{lemma}[theorem]{Lemma}
\newtheorem{proposition}[theorem]{Proposition}
\newtheorem{corollary}[theorem]{Corollary}
\theoremstyle{definition}
\newtheorem{definition}[theorem]{Definition}
\newtheorem{example}[theorem]{Example}

\newtheorem{conjecture}[theorem]{Conjecture}

\theoremstyle{remark}
\newtheorem{remark}[theorem]{Remark}
\numberwithin{equation}{section}
\usetikzlibrary{calc}
\pgfplotsset{compat=1.16}
\begin{document}

\title[Local approximate symmetry of Birkhoff-James orthogonality ]{Local approximate symmetry of Birkhoff-James orthogonality in normed linear spaces}

\author[J. Chmieli\'{n}ski, D. Khurana and D. Sain]{Jacek Chmieli\'{n}ski, Divya Khurana and Debmalya Sain}

\address[Chmieli\'{n}ski]{Department of Mathematics\\ Pedagogical
University of Krakow\\ Podchor\c{a}\.{z}ych 2, 30-084 Krak\'{o}w\\
Poland} \email{jacek.chmielinski@up.krakow.pl}

\address[Khurana]{Department of Mathematics\\ Indian Institute of Science\\ Bengaluru 560012\\ Karnataka \\India}
\email{divyakhurana11@gmail.com}

\address[Sain]{Department of Mathematics\\ Indian Institute of Science\\ Bengaluru 560012\\ Karnataka \\India}
\email{saindebmalya@gmail.com}

\thanks{}

\subjclass[2010]{Primary 46B20, Secondary 51F20, 52B15, 47L05}

\keywords{Birkhoff-James orthogonality; approximate Birkhoff-James
orthogonality; C-approximate symmetry; D-approximate symmetry}

\begin{abstract}
Two different notions of approximate Birkhoff-James orthogonality in
nor\-med linear spaces have been introduced by Dragomir and
Chmie\-li\'n\-ski. In the present paper we consider a global and a local
approximate symmetry of the Birkhoff-James orthogonality related to
each of the two definitions. We prove that the considered
orthogonality is approximately symmetric in the sense of Dragomir in
all finite-dimensional Banach spaces. For the other case, we prove
that for finite-dimensional polyhedral Banach spaces, the approximate
symmetry of the orthogonality is equivalent to some newly introduced
geometric property. Our investigations complement and extend the
scope of some recent results on a global approximate symmetry of the
Birkhoff-James orthogonality.
\end{abstract}

\maketitle
\section{Introduction}
The Birkhoff-James orthogonality is the most natural and well
studied notion of orthogonality in normed linear spaces. In general,
the Birkhoff-James orthogonality is not symmetric. Chmieli\'{n}ski
and W\'{o}jcik \cite{CW1} introduced a notion of  approximate
symmetry of the Birkhoff-James orthogonality in normed linear
spaces. It should be noted that the authors of \cite{CW1} considered
this notion in the global sense, the meaning of which will be clear
once we present the relevant definition in this section. In this
article, our motivation is to consider the corresponding local
version of the aforesaid concept. We also study the local version of
another standard notion of an approximate Birkhoff-James orthogonality
considered in \cite{D}. The advantage of considering the local
version is illustrated by obtaining some useful conclusions  in the
global case, separately for finite-dimensional polyhedral Banach
spaces and smooth Banach spaces.

Let us first establish the notations and the terminologies to be
used in the present article. Throughout the text, we use the symbols
$X,Y$ to denote real normed linear spaces. Given any two elements
$x,y\in X$, let
$\overline{xy}=\mbox{conv}\{x,y\}=\{(1-t)x+ty:t\in[0,1]\}$ denote
the closed line segment joining $x$ and $y$. By $B_X=\{x\in X:
\|x\|\leq1\}$ and $S_X=\{x\in X: \|x\|=1\}$ we denote the unit ball and
the unit sphere of $X$, respectively, and $B(x,\delta)$ denotes
the open unit ball in $X$  centered at $x$ and with the radius
$\delta>0$. The collection of all extreme points of $B_X$ will be denoted as $\mbox{Ext}\,B_X$.

Let $X^*$ denote the dual space of $X$. Given $0\not=x\in X$, $f\in
S_{X^*}$ is said to be a {\it supporting functional} at $x$ if
$f(x)=\|x\|$. Let $J(x)=\{f\in S_{X^*}: f(x)=\|x\|\}$, $0\not=x\in
X$, denote the collection of all supporting functionals at $x$. Note
that for each $0\not=x\in X$, the Hahn-Banach theorem ensures the
existence of at least one supporting functional at $x$.

An element $x\in S_X$ is said to be a {\it smooth point} if $J(x)=\{ f\}$
for some $f\in S_{X^*}$. Let $\mbox{sm}\,S_X$ denote the collection of all
smooth points of $S_X$. In particular if $\mbox{sm}~S_X=S_X$, then
$X$ is said to be a {\it smooth space}. Let $X$ be a Banach space with a norm $\|\ \|$. For every $\tau>0$, the modulus of smoothness is
defined by
\begin{align*}
\rho(\tau)=\sup\left\{\frac{\|x+\tau y\|+\|x-\tau y\|-2}{2}:\ x,y\in S_X\right\}.
\end{align*}
$(X,\|\ \|)$ is said to be a {\it uniformly smooth space} if
$\lim\limits_{\tau\rightarrow 0}\frac{\rho(\tau)}{\tau}=0$.

Let $X$ be a Banach space with a norm $\|\ \|$. For every
$\varepsilon\in(0,2]$, the modulus of convexity is defined by
\begin{align*}
\delta(\varepsilon)=\inf\left\{1-\frac{\|x+y\|}{2}:\ x,y  \in
B_X,\ \|x-y\|\geq \varepsilon\right\}.
\end{align*}
$(X,\|\ \|)$ is said to be {\it uniformly convex}  if
$\delta(\varepsilon)>0$ for all $\varepsilon\in(0,2]$.

It is well known that a Banach space $(X,\|\ \|)$ is uniformly
smooth if and only if its dual space $(X^*,\|\ \|^*)$ is
uniformly convex (see \cite{H} for more details).

For $x,y\in X$, we say that $x$ is {\it Birkhoff-James orthogonal} to $y$
\cite{B, J1}, written as $x\perp_B y$, if $\|x+\lambda y\|\geq
\|x\|$ for all $\lambda\in \mathbb{R}$. In \cite[Theorem 2.1]{J1},
James proved that if $0\not=x \in X$, $y\in X$, then $x\perp_B y$ if
and only if there exists $f\in J(x)$ such that $f(y)=0$. We will use
the notations $x^\perp=\{y\in X: x\perp_B y\}$ and $^\perp x=\{y\in
X: y\perp_B x\}$. Sain \cite{S1} characterized the Birkhoff-James
orthogonality of linear operators between finite-dimensional Banach
spaces by introducing the notions of the positive part of $x$,
denoted by $x^+$, and the negative part of $x$, denoted by $x^-$,
for an element $x\in X$. For any element $y\in X$, we say that $y\in
x^+$ $(y\in x^-)$ if $\|x+\lambda y\|\geq \|x\|$ for all $\lambda
\geq 0$ $(\lambda \leq 0)$. It is easy to see that $x^\perp=x^+\cap
x^-$.

Dragomir \cite{D} defined  an approximate Birkhoff-James orthogonality
as follows. Let $\varepsilon\in[0,1)$ and let $x,y\in X$; then $x$
is said to be approximately Birkhoff-James orthogonal to $y$ if
$\|x+\lambda y\|\geq (1-\varepsilon) \|x\|$ for all
$\lambda\in\mathbb{R}$. Later on, Chmieli\'{n}ski \cite{C} slightly
modified the definition given by Dragomir as follows. Let
$\varepsilon\in[0,1)$ and let $x,y\in X$. Then $x$ is said to be
approximately Birkhoff-James orthogonal to $y$, written as
$x\perp_D^\varepsilon y$, if and only if $\|x+\lambda y\|\geq
\sqrt{1-\varepsilon^2}\|x\|$ for all $\lambda\in\mathbb{R}$. Due to
this modification, in case of a Hilbert space, the present notion of
the approximate orthogonality coincides exactly with the usual notion of
the $\varepsilon$-orthogonality: $|\langle x,y \rangle|\leq
\varepsilon \|x\|\|y\|$. In \cite[Lemma 3.2]{MSP}, Mal et al. proved
that
\begin{align}\label{Dchar}
x\perp_D^\varepsilon y \quad \Leftrightarrow\quad \exists\, f\in
S_{X^*}:\  |f(x)|\geq \sqrt{1-\varepsilon^2}\|x\|
\mbox{~ and~} f(y)=0.
\end{align}

Chmieli\'{n}ski \cite{C} defined a variation of approximate
Birkhoff-James orthogonality. Given $x,y \in X$ and $\varepsilon\in
[0,1)$, $x$ is said to be approximately orthogonal to $y$, written
as $x\perp_B^\varepsilon y$, if $\|x+\lambda y\|^2\geq
\|x\|^2-2\varepsilon\|x\|\|\lambda y\|$ for all $\lambda \in
\mathbb{R}$. Later, in \cite[Theorems 2.2 and 2.3]{CSW}, Chmieli\'{n}ski et al.
gave two characterizations of this approximate orthogonality:
\begin{equation}\label{Bchar2}
x\perp_B^\varepsilon y \quad \Leftrightarrow \quad \exists\,z\in \mbox{span}\{x,y\}:\  x\perp_B z,\ \mbox{and}\ \|z-y\|\leq\varepsilon \|y\|;
\end{equation}
\begin{equation}\label{Bchar}
x\perp_B^\varepsilon y \quad \Leftrightarrow \quad \exists\,f\in
J(x):\ |f(y)|\leq \varepsilon \|y\|.
\end{equation}

Given $x,y\in X$ and $\varepsilon\in[0,1)$, we will write
$x\perp_{D}^{\varepsilon^*} y$ ($x\perp_{B}^{\varepsilon^*} y$) if
$x\perp_{D}^{\varepsilon} y$ ($x\perp_{B}^{\varepsilon} y$) but
$x\not\perp_D^{\varepsilon_1} y$  ($x\not\perp_B^{\varepsilon_1} y$)
for any $0\leq \varepsilon_1<\varepsilon$.

In general, orthogonality relation between two elements $x,y\in X$
need not be symmetric. In other words, for any two elements $x,y\in
X$, $x\perp_B y$ does not necessarily imply $y\perp_Bx$. James
\cite{J} proved that if dim $X\geq 3$ and the Birkhoff-James
orthogonality is symmetric, then the norm is induced by an inner
product.

In \cite{CW1}, Chmieli\'{n}ski and W\'{o}jcik defined the following
notion of  approximate symmetry of the  Birkhoff-James orthogonality
in a normed linear space.
\begin{definition}\label{C symmetry}
Let $X$ be a normed linear space. Then the Birkhoff-James
orthogonality is {\it approximately symmetric} if there exists
$\varepsilon\in[0,1)$ such that whenever $x,y\in X$ and $x\perp_B
y$, it follows that $y\perp_B^\varepsilon x$.
\end{definition}

The above definition is global in the sense that $\varepsilon$ is
independent of $x$ and $y$.

In this paper we will work with both of the above mentioned notions
of approximate Birkhoff-James orthogonality. To avoid any confusion
we will call the above notion of approximate symmetry an approximate symmetry of the Birkhoff-James orthogonality in the
sense of Chmieli\'{n}ski or shortly: {\it C-approximate symmetry} of the
Birkhoff-James orthogonality.

In \cite{CW1}, the authors gave an example of a Banach space where
the  Birkhoff-James orthogonality is not C-approximately symmetric.
In the present article we will study this example in more detail.
The following definition allows us to study local versions of the
C-approximate symmetry of the Birkhoff-James orthogonality.

\begin{definition}\label{local C}
Let $X$ be a normed linear space and let $x\in X$. We say that $x$
is {\it C-approximately left-symmetric} ({\it C-approximately right-symmetric})
if there exists $\varepsilon_x\in[0,1)$ such that whenever $y\in X$
and $x\perp_B y$ ($y\perp_B x$), it follows that
$y\perp_B^{\varepsilon_x} x$ ($x\perp_B^{\varepsilon_x} y$).

For $\mathcal{A}\subseteq X$ we say that the Birkhoff-James
orthogonality is C-approximately symmetric {\it on}  $\mathcal{A}$ if
there exists $\varepsilon\in[0,1)$ such that whenever $x$,
$y\in\mathcal{A}$ and $x\perp_B y$, it follows that
$y\perp_B^\varepsilon x$.

Let $\mathcal{A}\subseteq X$ and let $x\in S_X$. We say that $x$ is
C-approximately left-symmetric (C-approximately right-symmetric) {\it on}
$\mathcal{A}$ if there exists $\varepsilon_x\in[0,1)$ such that
whenever $y\in\mathcal{A}$ and $x\perp_B y$ ($y\perp_B x$), it
follows that $y\perp_B^{\varepsilon_x}x$
($x\perp_B^{\varepsilon_x}y$).
\end{definition}

Now, with respect to the Dragomir's definition, we define the
following analogous versions of approximate symmetry considered in
Definitions~\ref{C symmetry}~and~\ref{local C}.
\begin{definition}
Let $X$ be a normed linear space. We say that the Birkhoff-James
orthogonality is approximately symmetric in the sense of Dragomir,
shortly: the Birkhoff-James orthogonality is {\it D-approximately
symmetric}, if there exists $\varepsilon\in[0,1)$ such that whenever
$x,y\in X$ and $x\perp_B y$, it follows that $y\perp_D^\varepsilon
x$. For $x\in X$, we define $x$ to be D-approximately left-symmetric
(D-approximately right-symmetric), if there exists
$\varepsilon_x\in[0,1)$ such that whenever $y\in X$ and $x\perp_B y$
($y\perp_B x$), it follows that $y\perp_D^{\varepsilon_x} x$
($x\perp_D^{\varepsilon_x} y$).
\end{definition}

Observe that we can restrict ourselves to norm-one elements by
virtue of the homogeneity of all the notions of orthogonality and
approximate orthogonality introduced here.

To study the C-approximate left-symmetry and the C-approximate
right-symm-
etry of elements of a normed linear space $X$, we define
the following property. We say that the local property (P) holds for
$x\in S_X$ if
\begin{align*}
x^\perp\cap \mathscr{A}(x)=\emptyset,
\end{align*}
where $\mathscr{A}(x)$ is the collection of all those elements $y\in
S_X$ for which given any $f\in J(y)$, either $f$ or $-f$ is in
$J(x)$.

We say that the property (P) holds for a normed linear space $X$  if
the local property (P) holds for each $x\in S_X$, that is,
\begin{equation}
\mbox{for~ all~} x\in S_X:~ \mbox{the~local~ property~ (P)~holds}.
\tag{P}
\end{equation}

\bigskip

If $\mathcal{A}\subseteq S_X$ and $x\in S_X$, then we say that the local
property (P) holds for $x$ {\it on} $\mathcal{A}$ if $x^\perp \cap
\mathscr{A}(x)\cap\mathcal{A}=\emptyset$.

It follows trivially that the local property (P) holds for each
$x\in \mbox{sm}\,S_X$. We will prove that the local property (P) for an
$x\in S_X$ is equivalent to the C-approximate left-symmetry of $x$
in the local sense, that is, the local property (P) holds for $x\in
S_X$ if and only if for $y\in x^\perp\cap S_X$ there exists
$\varepsilon_{x,y}\in[0,1)$ such that $y\perp_B^{\varepsilon_{x,y}}
x$.

To study polyhedral Banach spaces, we recall the following
definitions from \cite{SPBB} which are relevant to our work:

\begin{definition}
Let $X$ be an $n$-dimensional Banach space. A {\it polyhedron} $P$ is a
non-empty compact subset of $X$ which is an intersection of finitely
many closed half-spaces of $X$, that means $P=\cap_{i=1}^r M_i$, where
$M_i$ are closed half-spaces in $X$ and $r\in\mathbb{N}$. The
{\it dimension} of the polyhedron $P$ is defined to be the dimension of
the subspace generated by the differences $x-y$ of vectors $x,y\in
P$.
\end{definition}

An $n$-dimensional Banach space $X$ is said to be a {\it polyhedral
Banach space} if  $B_X$ contains only finitely many extreme points,
or, equivalently, if $S_X$ is a polyhedron.

\begin{definition}
Let $X$ be an $n$-dimensional Banach space. A polyhedron $Q\subseteq
X$ is said to be a {\it face} of the polyhedron $P\subseteq X$ if either
$Q=P$ or if we can write $Q=P\cap \delta M$, where $M$ is a closed
half-space in $X$ containing $P$ and $\delta M$ denotes the boundary
of $M$. If the dimension of $Q$ is $i$, then $Q$ is called an
$i$-face of $P$. $(n-1)$-faces are called {\it facets} of $P$ and
$1$-faces of $P$ are called edges of $P$.
\end{definition}

\begin{definition}
Let $X$ be a finite-dimensional polyhedral Banach space and let $F$
be a facet of the unit ball $B_X$. A functional $f\in S_{X^*}$ is
said to be a {\it supporting functional} corresponding to the facet $F$ of
the unit ball $B_X$ if the following two conditions are satisfied:
\begin{itemize}
\item[(a)] $f$ attains its norm at some point $v$ of $F$.
\item[(b)] $F=(v+ker~f)\cap S_X$.
\end{itemize}
\end{definition}

It is easy to see that there is a unique hyperspace $H$ such that an
affine hyperplane parallel to $H$ contains the facet $F$ of the unit
ball $B_X$. Moreover, there exists a unique norm-one functional $f$,
such that $f$ attains its norm on $F$ and $ker~f=H$. In particular,
$f$ is a supporting functional to $B_X$ at every point of $F$.

Two elements $x,y\in\mbox{Ext}\,B_X$ of an $n$-dimensional polyhedral
Banach space $X$ are said to be {\it adjacent} if $\|tx+(1-t)y\|= 1$ for
all $t\in[0,1]$.

Given normed linear spaces $X,Y$, by $\mathcal{B}(X,Y)$
($\mathcal{K}(X,Y)$) we denote the space of all bounded (compact)
linear operators from $X$ to $Y$. A bounded linear operator $T\in
\mathcal{B}(X,Y)$ is said to {\it attain its norm} at $x\in S_X$ if
$\|Tx\|=\|T\|$. Let $M_T=\{x\in S_X:\|Tx\|=\|T\|\}$ be the
collection of all norm attaining elements of $T$. If $X$ is a
reflexive Banach space and $T\in\mathcal{K}(X,Y)$, then
$M_T\not=\emptyset$ (see \cite{Ba} for details).

The article is organized as follows. In Section 2, we study
the D-approximate symmetry of the Birkhoff-James orthogonality. For this notion the results we are able to obtain are of the highest level of generality. In particular, we prove that in all finite-dimen\-sio\-nal Banach spaces, the Birkhoff-James
orthogonality is always D-appro\-xi\-ma\-te\-ly symmetric.

In Section 3, we study the C-approximate symmetry of the Birkhoff-James
orthogonality. We prove that in finite-dimensional polyhedral Banach
spaces, the C-approximate symmetry of the Birkhoff-James  orthogonality
is equivalent to the local property (P) of all elements of
${\rm Ext}\,B_X$. Apparently, the results in this section are less general than in Section 2. It is caused by the fact that the notion of the C-approximate symmetry essentially differs from the D-approximate one and not all properties remain true. Thus we need to use more subtle methods which usually involve additional assumptions.

In Section 4, we study the C-approximate symmetry of the Birkhoff-James orthogonality for two-dimensional Banach spaces. Even in this case, establishing a satisfactory characterization of the C-approximate symmetry is challenging. To this aim, we introduce a new property, namely property (P1). We show that for any finite-dimensional polyhedral Banach space with property (P1), local property (P) also holds for each element.  We also show that the converse is true for any two-dimensional polyhedral Banach spaces but in general it need not be true. We show that in a two-dimensional regular polyhedral Banach space with $2n$ vertices, where $n\geq 3$, the Birkhoff-James orthogonality is C-approximately
symmetric. We provide an example to show that the regularity
condition in this case cannot be dropped.

\section{D-approximate symmetry of the Birkhoff-James orthogonality}

In \cite{CW1}, Chmieli\'{n}ski and W\'{o}jcik proved that in
uniformly convex Banach spaces and finite-dimensional smooth Banach
spaces, the Birkhoff-James orthogonality is C-approximately
symmetric. Our main aim in this section is to prove that for any
finite-dimensional Banach space, the Birkhoff-James orthogonality is
D-approximately symmetric. To achieve this aim, we first prove the
following results.

\begin{theorem}\label{neighbourhood}
Let $X$ be a normed linear space and let $x,y\in S_X$ with
$x\perp_D^{\varepsilon} y$ for some $\varepsilon \in [0,1)$ . Then
there exist $\varepsilon_1,\varepsilon_2>0$, $\varepsilon_3\in
(0,1)$ such that $z\perp_D^{\varepsilon_3} w$ for all $z\in
B(x,\varepsilon_1)\cap S_X$, $w\in B(y,\varepsilon_2)\cap S_X$.
\end{theorem}

\begin{proof}
Let $\varepsilon_1> 0$ be such that
$\sqrt{1-\varepsilon^2}-\varepsilon_1> 0$. If $z\in
B(x,\varepsilon_1)\cap S_X$, then
\begin{align*}
\|z+\lambda y\|=\|z-x+x+\lambda y\|\geq \|x+\lambda y\|-\|x-z\|\geq
\sqrt{1-\varepsilon^2}-\varepsilon_1.
\end{align*}

Thus for all $z\in B(x,\varepsilon_1)\cap S_X$, we have,
$z\perp_D^{\delta} y$ where $\delta$ is such that
$\sqrt{1-\varepsilon^2}-\varepsilon_1=\sqrt{1-\delta^2}$.

If $|\lambda|\geq 2$, then for any $z_1,z_2\in S_X$, we have,
$\|z_1+\lambda z_2\|\geq |\lambda|-1\geq 1\geq 1-\beta$ for all
$\beta\in [0,1)$.

Choose $\varepsilon_2>0$ such that  $\sqrt{1-\delta
^2}-2\varepsilon_2>0$. Now, if $\lambda\in \mathbb{R}$ with
$|\lambda|< 2$, then for any $z\in B(x,\varepsilon_1)\cap S_X$ and
$w\in B(y,\varepsilon_2)\cap S_X$, we have,
\begin{align*}
\|z+\lambda w\|=\|z+\lambda y+\lambda w-\lambda y\|\geq \|z+\lambda
y\|-|\lambda|\| y-w\|> \sqrt{1-\delta^2}-2\varepsilon_2.
\end{align*}
Thus for all $z\in B(x,\varepsilon_1)\cap S_X$ and $w\in
B(y,\varepsilon_2)\cap S_X$, we have, $z\perp_D^{\varepsilon_3} w$
where  $\varepsilon_3$ is such that
$\sqrt{1-\delta^2}-2\varepsilon_2=\sqrt{1-\varepsilon_3^2}$.
\end{proof}

Our next result shows that given any two linearly independent
elements $x,y\in S_X$ of a normed linear space $X$, we can always
find an $\varepsilon \in[0,1)$ (depending on $x$ and $y$) such that
$x\perp_D^\varepsilon y$.

\begin{proposition}\label{approximate orthogonality}
Let $X$ be a normed linear space and let $x,y\in S_X$ with
$x\not=\pm y$. Then there exists $\varepsilon_{x,y}\in[0,1)$ such
that $x\perp_D^{\varepsilon_{x,y}} y$.
\end{proposition}

\begin{proof}
Since $x,y\in S_X$ and $x\not=\pm y$, it follows that $x$, $y$ are
linearly independent. Let $X_0=\mbox{span}\{x,y\}$ and let
$\{x^*,y^*\}\subseteq {X_0^*}$ be such that $\{x, y; x^*,y^*\}$ is a
biorthogonal system in $X_0$, where $x^*(x)=y^*(y)=1,~
x^*(y)=y^*(x)=0$. Now, if we take $f=\frac{x^*}{\|x^*\|}$, then $f\in
S_{X_0^*}$, $f(x)=\frac{1}{\|x^*\|}$ and $f(y)=0$. Let $\hat{f}$ be
a Hahn-Banach extension of $f$ to $X^*$. Then $\hat{f}\in S_{X^*}$,
$\hat{f}(x)=\frac{1}{\|x^*\|}$ and $\hat{f}(y)=0$. If
$\frac{1}{\|x^*\|}> 1$, then for all $\varepsilon\in[0,1)$, we have,
$\hat{f}(x)\geq \sqrt{1-\varepsilon^2}$. If $\frac{1}{\|x^*\|}\leq
1$, then we can find $\varepsilon\in[0,1)$ such that $\hat{f}(x)\geq
\sqrt{1-\varepsilon^2}$. Thus (\ref{Dchar}) implies that for given
$x,y\in S_X$, with $x\not=\pm y$, there exists
$\varepsilon_{x,y}\in[0,1)$ such that $x\perp_D^{\varepsilon_{x,y}}
y$.
\end{proof}

\begin{theorem}\label{theorem approximate symmetric}
Let $X$ be a finite-dimensional Banach space. Then the
Birkhoff-James orthogonality is D-approximately symmetric in $X$.
\end{theorem}
\begin{proof}
Let $x\in S_X$ and let $y\in x^\perp\cap S_X$. Then by
Proposition~\ref{approximate orthogonality}, there exists
$\varepsilon_{x,y}\in [0,1)$ such that
$y\perp_D^{\varepsilon_{x,y}} x$. Let $\varepsilon^*_{x,y}$ be the
infimum of all such $\varepsilon_{x,y}$.  We claim that
$\varepsilon:=\underset{x\in S_X}{\sup}~\underset{y\in x^\perp\cap
S_X}{\sup} \varepsilon^*_{x,y}<1$. If $\varepsilon= 1$, then we can
choose $\{x_n\}$, $\{y_n\}\subseteq S_X$, $\varepsilon_n\nearrow 1$
such that $x_n\perp_B y_n$ and $y_n\perp_{D}^{\varepsilon_n^*}x_n$.
Since $S_X$ is compact, there exist convergent sub-sequences of
$\{x_n\}$, $\{y_n\}$ which we again denote by $\{x_n\}$ and
$\{y_n\}$, respectively. Let $x_0$, $y_0\in S_X$ be such that
$x_n\longrightarrow x_0$ and $y_n\longrightarrow y_0$. Then by
continuity of the norm it follows that $y_0\in x_0^\perp\cap S_X$.
Now from Proposition~\ref{approximate orthogonality}, it follows
that $y_0\perp_{D}^{\varepsilon_0} x_0$ for some $\varepsilon_0 \in
[0,1)$. Using Theorem~\ref{neighbourhood}, we can find
$\varepsilon_1$, $\varepsilon_2>0$ and $\varepsilon_3\in(0,1)$ such
that $w\perp_{D}^{\varepsilon_3}z$ for all $z\in
B(x_0,\varepsilon_1)\cap S_X$ and $w\in B(y_0,\varepsilon_2)\cap
S_X$. Thus we can find $m\in\mathbb{N}$ such that
$y_n\perp_{D}^{\varepsilon_3}x_n$ for all $n\geq m$. This leads to a
contradiction as $y_n\perp_{D}^{\varepsilon_n^*}x$ for
$\varepsilon_n\nearrow 1$. Thus $\varepsilon<1$ and the
Birkhoff-James orthogonality is D-approximately symmetric in $X$.
\end{proof}

\begin{remark}
It follows from the above theorem that each element
of a finite-dimensional Banach space is both D-approximately
left-symmetric and D-appro\-xi\-ma\-te\-ly right-symmetric.
\end{remark}

We will use the following result from \cite{SPM2} in the proof of
the next result.

\begin{theorem}\cite[Theorem 2.1]{SPM2}\label{operator dragomir}
Let $X$ be a reflexive Banach space and let $Y$ be a normed linear
space. Let $T,A\in\mathcal{K}(X,Y)$ with $\|T\|=\|A\|=1$. Then
$T\perp_D^\varepsilon A$ for $\varepsilon\in[0,1)$ if and only if
either $(a)$ or $(b)$ holds.
\begin{itemize}
\item[(a)] There exists $x\in M_T$ such that $Ax\in(Tx)^+$ and for
each $\lambda\in(-1-\sqrt{1-\varepsilon^2},
-1+\sqrt{1-\varepsilon^2} )$, there exists  $x_\lambda\in S_X$ such
that $\|Tx_\lambda+\lambda Ax_\lambda\|\geq \sqrt{1-\varepsilon^2}$.

\item[(b)] There exists $y\in M_T$ such that $Ay\in(Ty)^-$ and for
each $\lambda\in(1-\sqrt{1-\varepsilon^2}, 1+\sqrt{1-\varepsilon^2}
)$, there exists  $y_\lambda\in S_X$ such that $\|Ty_\lambda+\lambda
Ay_\lambda\|\geq \sqrt{1-\varepsilon^2}$.
\end{itemize}
\end{theorem}

Let $X$ be a reflexive Banach space and let $Y$ be a normed linear
space. Let $T,A\in S_{\mathcal{K}(X,Y)}$ be such that $T\perp_B
A$. Then by Proposition~\ref{approximate orthogonality}, there exists
$\varepsilon\in[0,1)$ such that $A\perp_D^\varepsilon T$. We now
estimate the infimum of such $\varepsilon$'s.

\begin{theorem}\label{operator orthogonality}
Let $X$ be a reflexive Banach space and $Y$ a normed linear
space. Suppose that $T,A\in\mathcal{K}(X,Y)$ with $\|T\|=\|A\|=1$ and that the set
$\mathcal{A}=\{x\in S_X: Tx\not=\lambda
Ax~\mbox{for~all}~\lambda\in\mathbb{R}\}$ is nonempty. If $A
\perp_B T$, then $T\perp_D^\varepsilon A$, where
$\sqrt{1-\varepsilon^2}=\sup_{x\in
\mathcal{A}}\inf_{\lambda\in\mathbb{R}}\|Tx+\lambda Ax\|$.
\end{theorem}

\begin{proof}
Let $x_0\in \mathcal{A}$. Then $Tx_0\not=0$ and by continuity of the
function $f(\lambda)=\|Tx_0+\lambda Ax_0\|$, $\lambda\in\mathbb{R}$
and the fact that $f(\lambda)\longrightarrow \infty$ as
$\lambda\longrightarrow\pm \infty$,  it follows that
$\inf_\lambda\|Tx_0+\lambda Ax_0\|>0$. Also,
$\inf_\lambda\|Tx_0+\lambda Ax_0\|\leq \|Tx_0\|\leq 1$. Let
$\varepsilon_{x_0}\in[0,1)$ be such that $\inf_\lambda\|Tx_0+\lambda
Ax_0\|=\sqrt{1-\varepsilon_{x_0}^2}$. If $x\in X$, then it follows
from \cite[Proposition 2.1]{S1} that either $Ax\in (Tx)^+$ or $Ax\in
(Tx)^-$. Since $X$ is a reflexive Banach space and
$T\in\mathcal{K}(X,Y)$, it follows that $M_T\not=\emptyset$. Now, by
using Theorem~\ref{operator dragomir}, we get
$T\perp_D^{\varepsilon_{x_0}} A$. If we fix $\alpha=\sup_{x\in
\mathcal{A}}\inf_{\lambda\in\mathbb{R}}\|Tx+\lambda Ax\|$, then
clearly $\alpha \in (0,1]$. Let $\varepsilon\in[0,1)$ be such that
$\alpha=\sqrt{1-\varepsilon^2}$. Then $T\perp_D^\varepsilon A$ and
this completes the proof.
\end{proof}

\begin{remark}
The proof of the above theorem suggests that if $x_0\in\mathcal{A}$
and $\inf_\lambda\|Tx_0+\lambda Ax_0\|=\sqrt{1-\varepsilon_{x_0}^2}$,
then $T\perp_D^{\varepsilon_{x_0}} A$. Thus
$\sqrt{1-\varepsilon^2}=\sup_{x\in
\mathcal{A}}\inf_{\lambda\in\mathbb{R}}\|Tx\\+\lambda Ax\|$ provides
the best possible estimate for $\varepsilon\in[0,1)$ such that
$T\perp_D^\varepsilon A$.
\end{remark}

As an application of the above theorem, for finite-dimensional
spaces $X$, $Y$ and operators $T$, $A\in S_{\mathcal{B}(X,Y)}$ with
$A\perp_B T$, we now obtain an estimate of $\varepsilon$ such that
$T\perp_D^\varepsilon A$.

\begin{theorem}
Let $X,Y$ be finite-dimensional Banach spaces. Let
$T,A\in\mathcal{B}(X,Y)$ with $\|T\|=\|A\|=1$ and let
$\mathcal{A}=\{x\in S_X: Tx\not=\lambda
Ax~\mbox{for~all}~\lambda\in\mathbb{R}\}$. If $A \perp_B T$, then
$T\perp_D^\varepsilon A$, where $\sqrt{1-\varepsilon^2}=\sup_{x\in
\mathcal{A}}\inf_{\lambda\in\mathbb{R}}\|Tx+\lambda Ax\|$.
\end{theorem}

\begin{proof}
In order to apply Theorem~\ref{operator orthogonality}, we need to
show that $\mathcal{A}\not=\emptyset$. Suppose on the contrary that
$\mathcal{A}=\emptyset$. Then for each $x\in S_X$, there exists
$\lambda_x\in \mathbb{R}$ such that $Tx=\lambda_x Ax$. Clearly, $A
\perp_B T$ implies that there does not exist $\lambda\in\mathbb{R}$
such that $Tx=\lambda Ax$ for all $x\in X$. We now consider the
following two cases.

Let $\mbox{rank}~ A\geq 2$ and let $\{Ax_1,\ldots, Ax_k\}$ be a
basis for $\mbox{range} ~A$,  where $x_1,\ldots,x_k\in S_X$, $2\leq
k \leq n$, where $\mbox{dim}~ X=n$. Let $\{x_1,\ldots,
x_k,x_{k+1},\ldots, x_n\}$ be a basis for $X$, where
$\{x_{k+1},\ldots, x_n\}\subseteq S_X$ is a basis for $\mbox{ker}~
A$.

In this case there are following two possibilities:

$(i)$ there exist $1\leq i,j\leq k$ such that $Tx_i=\lambda_{x_i}
Ax_i$ and $Tx_j=\lambda_{x_j} Ax_j$ for
$\lambda_{x_i}\not=\lambda_{x_j}$,

$(ii)$  there exists a $\lambda\in\mathbb{R}$ such that
$Tx_i=\lambda Ax_i$ for each $1\leq i\leq k$.

First consider the case $(i)$. In this case
\begin{align*}
T\left(\frac{x_i+x_j}{\|x_i+x_j\|}\right)=\frac{\lambda_{x_i}Ax_i+\lambda_{x_j}Ax_j}{\|x_i+x_j\|}.
\end{align*}

Using the assumption $\mathcal{A}=\emptyset$, let
$\lambda\in\mathbb{R}$ be such that
\begin{align*}
T\left(\frac{x_i+x_j}{\|x_i+x_j\|}\right)=\lambda
A\left(\frac{x_i+x_j}{\|x_i+x_j\|}\right).
\end{align*}

Thus $(\lambda_{x_i}-\lambda)Ax_i+(\lambda_{x_j}-\lambda)Ax_j=0$ and
this proves that $\lambda=\lambda_{x_i}=\lambda_{x_j}$. This leads
to a contradiction as $\lambda_{x_i}\not=\lambda_{x_j}$.

Now, we will consider the case $(ii)$ as above. In this case
$x_{i_0}\not\in \mbox{ker}~ T$ for at least one $i_0$, $k+1\leq
i_0\leq n$, otherwise $T=\lambda A$. Clearly $Tx_{i_0}\not=\lambda
Ax_{i_0}$ for all $\lambda\in\mathbb{R}$. This contradicts that
$\mathcal{A}=\emptyset$. Thus if $\mbox{rank}~ A\geq 2$, then
$\mathcal{A}\not=\emptyset$.

Now, consider the case when $\mbox{rank}~ A=1$. Let $\mbox{range}~
A=\mbox{span}~\{Ax_1\}$ where $x_1\in S_X$ and
$\{x_2,\ldots,x_n\}\subseteq S_X$ be a basis for $\mbox{ker} ~A$. By
the assumption $\mathcal{A}=\emptyset$ and thus $Tx_1=\lambda_{x_1}
Ax_1$ for some $\lambda_{x_1}\in \mathbb{R}$. Clearly, $A\perp_B T$
implies $x_{i_0}\not\in \mbox{ker}~ T$ for at least one $i_0$,
$2\leq i_0\leq n$. This implies $Tx_{i_0}\not=\lambda Ax_{i_0}$ for
all $\lambda\in\mathbb{R}$. This contradicts that
$\mathcal{A}=\emptyset$ and thus in this case also
$\mathcal{A}\not=\emptyset$. Now, the result follows from
Theorem~\ref{operator orthogonality}.

\end{proof}

The above theorem can be extended to compact operators on a
reflexive Banach space, under the additional assumption of
injectivity of $A$ or $T$.

\begin{theorem}
Let $X$ be a reflexive Banach space and  $Y$ any normed linear
space. Assume that $T,A\in \mathcal{K}(X,Y)$ with $\|T\|=\|A\|=1$ and either
$A$ or $T$ is one to one operator. Define $\mathcal{A}=\{x\in S_X:
Tx\not=\lambda Ax~\mbox{for~all}~\lambda\in\mathbb{R}\}$. If $A
\perp_B T$, then $T\perp_D^\varepsilon A$, where
$\sqrt{1-\varepsilon^2}=\sup_{x\in
\mathcal{A}}\inf_{\lambda\in\mathbb{R}}\|Tx+\lambda Ax\|$.
\end{theorem}

\begin{proof}
To prove the result we need to show that
$\mathcal{A}\not=\emptyset$. Suppose on the contrary that
$\mathcal{A}=\emptyset$. Then for each $x\in S_X$, there exists
$\lambda_x\in \mathbb{R}$ such that $Tx=\lambda_x Ax$. Clearly $A
\perp_B T$ implies that there exist $x,y\in S_X$ such that
\begin{equation}\label{lambda}
Tx=\lambda_x Ax\quad \mbox{~and~}\quad Ty=\lambda_y Ay
\end{equation}
for $\lambda_x\not=\lambda_y$. This implies $x$ and $y$ are linearly
independent in $X$.

Let $\lambda\in \mathbb{R}$ be such that
\begin{align}\label{lambda1}
T\left(\frac{x+y}{\|x+y\|}\right)=\lambda
A\left(\frac{x+y}{\|x+y\|}\right).
\end{align}
Also, we have,
\begin{align}\label{lambda2}
T\left(\frac{x+y}{\|x+y\|}\right)=
\frac{\lambda_xAx+\lambda_yAy}{\|x+y\|}.
\end{align}

Let us first assume that $A$ is one to one operator. Now, using
(\ref{lambda1}), (\ref{lambda2}) we have
$A((\lambda_x-\lambda)x+(\lambda_y-\lambda)y)=0$ and using the
assumption that $A$ is one to one we get
$(\lambda_x-\lambda)x+(\lambda_y-\lambda)y=0$. It follows from the
linear independence of $x,y$ that $\lambda_x=\lambda=\lambda_y$. But
this leads to a contradiction as $\lambda_x\not=\lambda_y$. This
implies that $\mathcal{A}\not=\emptyset$. Thus in this case the
result follows from Theorem~\ref{operator orthogonality}.

Now, we assume that $T$ is one to one operator. It follows from this
assumption on $T$ that in (\ref{lambda}) $\lambda_x,\lambda_y\not=0$
and also in (\ref{lambda1}), we have $\lambda\not=0$. After
rewriting (\ref{lambda1}) and using (\ref{lambda}), we get,
\begin{align*}
\frac{1}{\lambda}T\left(\frac{x+y}{\|x+y\|}\right)=
A\left(\frac{x+y}{\|x+y\|}\right),
\end{align*}

\begin{align*}
\frac{\frac{1}{\lambda_x}Tx+\frac{1}{\lambda_y}Ty}{\|x+y\|}=A\left(\frac{x+y}{\|x+y\|}\right).
\end{align*}

Thus
$T((\frac{1}{\lambda_x}-\frac{1}{\lambda})x+(\frac{1}{\lambda_y}-\frac{1}{\lambda})y)=0$
and using the assumption that $T$ is one to one we get
$(\frac{1}{\lambda_x}-\frac{1}{\lambda})x+(\frac{1}{\lambda_y}-\frac{1}{\lambda})y=0$.
Now, the result follows from the similar arguments as those used in
the previous case.
\end{proof}

\section{C-approximate symmetry of the Birkhoff-James orthogonality}

It was observed in \cite{CW1}, that in
$(\mathbb{R}^2,\|~\|_\infty)$ the Birkhoff-James orthogonality is
not C-approximately symmetric. In the following proposition we study
the C-approximate left-symmetry and the C-approximate right-symmetry
of elements of $(\mathbb{R}^n,\|~\|_\infty)$ in detail. In
particular, the following result illustrates that in the local
sense, the C-approximate left-symmetry is not equivalent to
the C-approximate right-symmetry of the Birkhoff-James orthogonality. It
is well known that the dual of $(\mathbb{R}^n,\|~\|_\infty)$ can be
identified with $(\mathbb{R}^n,\|~\|_1)$, where the dual action is
given by $f(x)=\sum\limits_{i=1}^{n}f_i x_i$ for all $x=(x_1,\ldots,
x_n)\in(\mathbb{R}^n,\|~\|_\infty)$ and $f=(f_1,\ldots,
f_n)\in(\mathbb{R}^n,\|~\|_1)$. If $t\in \mathbb{R}$, then ${\rm sgn}\,t$
denotes the sign function, that is, ${\rm sgn}\,t=\frac{t}{|t|}$ for
$t\not=0$ and ${\rm sgn}\,0=0$.

\begin{proposition}\label{ellinfinity}
Let $X=(\mathbb{R}^n,\|~\|_\infty)$. Then

$(i)$ any smooth point $x\in S_X$ is C-approximately
left-symmetric but not C-approximately right-symmetric;

$(ii)$ any extreme point $x$ of $S_X$ is C-approximately
right-symmetric but not C-approximately left-symmetric.
\end{proposition}

\begin{proof}

Observe that from the symmetry of $S_X$, it is sufficient to prove
the result for any one of the extreme points and smooth points of
$S_X$.

$(i)$ Let $x=(1,x_2,\ldots,x_n)\in S_X$ be a smooth point. Then
$|x_i|<1$ for all $2\leq i\leq n$ and $J(x)=\{f\}$ where
$f=(1,0,\ldots,0)\in S_{X^*}$. Let $y=(y_1,y_2,\ldots,y_n)\in S_X$
be such that $x\perp_B y$. Then by using $(\ref{Bchar})$, it follows
that $y_1=0$. As $y\in S_X$, there exists $2\leq i_0\leq n$ such
that $|y_{i_0}|=1$. Let $g=(0,0,\ldots,0,{y_{i_0}},0,\ldots,0)\in
S_{X^*}$, where $y_{i_0}$ is the $i_0$-$th$ co-ordinate. Then $g\in
J(y)$ and $|g(x)|=|x_{i_0}|<1$. Thus $(\ref{Bchar})$ implies that
$y\perp_B^{\varepsilon_0} x$, where $\varepsilon_0=|x_{i_0}|$. Now,
if we take $\varepsilon=\underset{2\leq i\leq n}{\max}|x_i|$, then
$\varepsilon\in [0,1)$ and $y\perp_B^\varepsilon x$ whenever
$x\perp_B y$. Hence $x$ is C-approximately left-symmetric.

Now, we show that $x$ is not C-approximately right-symmetric. If
$x_i=0$ for all $2\leq i\leq n$, then $z\perp_B x$, where
$z=(1,1,\ldots,1)$. As $|f(z)|=1$,  there does not exist any
$\varepsilon\in[0,1)$ such that $x\perp_B^\varepsilon z$. Without
loss of generality we now assume that $x_2\not=0$. Let
$w=(1,-{\rm sgn}\,x_2,x_3,\ldots,x_n)\in S_X$. Then for any $\lambda\geq
0$ we have $\|w+\lambda x\|\geq |1+\lambda|\geq 1$.
Also, for any $\lambda<0$, we have
$\|w+\lambda x\|\geq |-{\rm sgn}\,x_2+\lambda x_2|> 1$.

This shows that $w\perp_B x$. As $|f(w)|=1$, there does not exist
any $\varepsilon\in[0,1)$ such that $x\perp_B^\varepsilon w$. Thus,
$(\ref{Bchar})$ implies that $x$ is not C-approximately
right-symmetric. Figure 1, given below, illustrates this situation
for $n=2$.

$(ii)$  Consider $x=(1,1,\ldots,1)\in \mbox{Ext}\,B_X$. It follows
from the arguments of $(i)$ that $x$ is not C-approximately
left-symmetric.

We now prove that $x$ is C-approximately right-symmetric. Consider
$y=(y_1,y_2,\ldots, y_n)\in S_X$  such that $y\perp_B x$. Since
$y\in S_X$, there exists $1\leq i\leq n$ such that $|y_{i}|=1$. Let
$\{i_1,i_2,\ldots,i_k\}\subseteq\{1,2,\ldots,n\}$ be a maximal
subset such that $|y_{i_j}|=1$ for $1\leq j\leq k$. We now claim
that $k>1$. Suppose on the contrary that $k=1$. Then $y\in
\mbox{sm}~S_X$, $J(y)=\{f\}$, where
$f=(0,0,\ldots,0,1,0,\ldots,0)\in S_{X^*}$ and $1$ is the $i_1$-$th$
co-ordinate. But $f(x)\not=0$ and this contradicts that $y\perp_B
x$.

We now claim that there exist $1\leq l\not=m\leq k$ such that
$y_{i_l}=-y_{i_m}$. If $y_{i_l}=y_{i_m}$ for all $1\leq l,m\leq k$, 
then for sufficiently small absolute value $\lambda$, it is easy to
see that $\|y+\lambda x\|=|y_{i_1}+\lambda|$. This clearly
contradicts that $y\perp_B x$ and hence there exist $1\leq
l\not=m\leq k$ such that $y_{i_l}=-y_{i_m}$. Now, if we take
$g=(0,0,\ldots,0,{\frac{1}{2}},0,\ldots,0,
{\frac{1}{2}},0,\ldots,0)\in S_{X^*}$, where $\frac{1}{2}$ is at
$i_1$-$th$ and $i_m$-$th$ co-ordinates. Then $g\in J(x)$ and
$g(y)=0$. This shows that $x$ is right-symmetric and hence
C-approximately right-symmetric. Figure 2, given below, illustrates
this situation for $n=2$.

\begin{tikzpicture}[scale=1.3]

\draw (-1.2,1.2)  node {$y$};

\draw (1.2,1.2)  node {$x$};

\draw (-1.3,-1.2)  node {$-x$};

\draw (0.7,-1.2)  node {$-y$};

\draw [fill] (-1,1) circle [radius=.05];

\draw [fill] (1,1) circle [radius=.05];

\draw [fill] (-1,-1) circle [radius=.05];

\draw [fill] (1,-1) circle [radius=.05];

\draw[thick] (-1,1)--(1,1) node[right]{};

\draw[thick] (-1,-1)--(1,-1) node[right]{};

\draw[thick] (1,1)--(1,-1) node[right]{};

\draw[thick] (-1,1)--(-1,-1) node[right]{};

\draw[thick, black,dashed] (-2,0)--(2,0) node[above]{};

\draw[thick,black,dashed] (0,-2)--(0,2) node[above]{};

\draw[thick,red, dashed] (-1.7,1.7)--(1.7,-1.7) node[right]{};

\draw[thick,red, dashed] (2.5,-0.5)--(-0.5,2.5) node[right]{};

\draw (-0.4,2.1) node {$g$};

\draw (-1.9,1.4) node {$\mbox{ker}~g$};

\draw[thick] (-1,1)--(1,1) node[right]{};

\draw[thick] (-1,-1)--(-1,-1) node[right]{};

\draw (-6.3,1.2)  node {$-w$};

\draw (-3.8,1.2)  node {$z$};

\draw (-6.3,-1.2)  node {$-z$};

\draw (-4.2,-1.2)  node {$w$};

\draw [fill] (-6,1) circle [radius=.05];

\draw [fill] (-4,1) circle [radius=.05];

\draw [fill] (-6,-1) circle [radius=.05];

\draw [fill] (-4,-1) circle [radius=.05];

\draw[thick] (-6,1)--(-4,1) node[right]{};

\draw[thick] (-6,-1)--(-4,-1) node[right]{};

\draw[thick] (-4,1)--(-4,-1) node[right]{};

\draw[thick] (-6,1)--(-6,-1) node[right]{};

\draw[thick, blue,dashed] (-7.2,1)--(-2.8,1) node[right]{};

\draw (-7,1.2) node{$g$};

\draw[thick,blue, dashed] (-7.2,-1)--(-2.8,-1) node[right]{};

\draw (-7,-0.8) node{$-g$};

\draw[thick,blue, dashed] (-7.2,0)--(-2.8,0) node[above]{};

\draw (-7,0.2) node{$\mbox{ker}~ g$};

\draw[thick, red, dashed] (-5,-2)--(-5,2) node[above]{};

\draw (-3.8,0.4)  node {$x$};

\draw (-6.3,0.4)  node {$-x$};

\draw (-5.2,1.2)  node {$y$};

\draw (-5.4,-1.2)  node {$-y$};

\draw [fill] (-4,0.4) circle [radius=.05];

\draw [fill] (-6,0.4) circle [radius=.05];

\draw [fill] (-5,1) circle [radius=.05];

\draw [fill] (-5,-1) circle [radius=.05];

\draw[thick, red, dashed] (-4,2)--(-4,-2) node[right]{};

\draw (-3.8,-1.6) node {$f$};

\draw[thick, red, dashed] (-6,2)--(-6,-2) node[right]{};

\draw (-5.75,-1.6) node {$-f$};

\draw (-4.6,-1.6) node {$\mbox{ker} ~f$};

\draw (-5,-2.5) node {Figure 1};

\draw (0,-2.5)  node {Figure 2};

\end{tikzpicture}

\end{proof}

\bigskip

\begin{remark}
The proof of the above proposition suggests that in
$X=(\mathbb{R}^n,\|~\|_\infty)$ the C-approximate right-symmetry
(the C-approximate left-symmetry) of any $x\in \mbox{sm}~S_X$ ($x\in
\mbox{Ext}\,B_X$) fails because there exists $y\in$$^\perp x\cap S_X$
($y\in x^\perp \cap S_X$) such that either $f\in J(y)$ or $-f \in
J(y)$ ($f\in J(x)$ or $-f \in J(x)$) where $J(x)=\{f\}$ ($f\in
J(y)$).
\end{remark}

The above remark is the main motivation behind considering the local
property (P) for $x\in S_X$ introduced in the first section. Recall
that the local property (P) holds for $x\in S_X$ if
\begin{align*}
x^\perp\cap \mathscr{A}(x)=\emptyset,
\end{align*}
where $\mathscr{A}(x)$ is the collection of all those elements $y\in
S_X$ for which given any $f\in J(y)$, either $f$ or $-f$ is in
$J(x)$. Also, recall that the property (P) holds for a normed linear
space $X$ if the local property (P) holds for each $x\in S_X$.

We now show that in finite-dimensional Banach spaces, the local
property (P) for all elements of $\mbox{Ext}\,B_X$ implies the
property (P) globally for  $X$.

\begin{theorem}\label{extreme points}
Let $X$ be a finite-dimensional Banach space and suppose that the
local property {\rm (P)} holds for each $x\in{\rm Ext}\,B_X$. Then the
local property {\rm (P)} holds for each $y\in S_X$.
\end{theorem}

\begin{proof}
It follows easily that in any normed linear space, the local
property (P) holds for each smooth point. Thus, to prove the result
we need to show that the local property (P) holds for any $y\in
S_X\setminus(\mbox{sm}~S_X\cup\mbox{Ext}\,B_X)$. Let $y\in
S_X\setminus(\mbox{sm}~S_X\cup\mbox{Ext}\,B_X)$. Since $S_X$ is
contained in the convex hull of $\mbox{Ext}\,B_X$, let $x_1,\ldots, x_k\in
\mbox{Ext}\,B_X$, $k\leq |\mbox{Ext}\,B_X|$, be such that
$y=\sum_{i=1}^k\alpha_ix_i$, $\alpha_i>0$ for all $1\leq i\leq k$
and $\sum_{i=1}^k\alpha_i=1$.

Now, we claim that if $f\in J(y)$, then $f\in J(x_i)$ for all $1\leq
i\leq k$. Clearly, $|f(x_i)|\leq 1$ for all $1\leq i \leq k$.
Suppose on the contrary that $f(x_j)<1$ for some $1\leq  j\leq k$.
Then
\begin{align*}
1=f(y)=\sum_{i=1}^k\alpha_if(x_i)<\sum_{i=1}^k \alpha_i=1.
\end{align*}
This clearly leads to a contradiction and thus if $f\in J(y)$, then
$f\in J(x_i)$ for all $1\leq i\leq k$.

Let $z\in y^\perp\cap S_X$. Then there exists $g\in J(y)$ such that
$g(z)=0$. But $g\in J(x_i)$ for all $1\leq i\leq k$; this gives
$x_i\perp_B z$ for all $1\leq i\leq k$.

We now claim that there exists some $g_0\in J(z)$ such that
$|g_0(y)|<1$. Suppose on the contrary that for any $g\in J(z)$ we
have $|g(y)|=1$, that is, for any $g\in J(z)$ either $g$ or $-g$ is
in $J(y)$. Thus for any $g\in J(z)$ either $g$ or $-g$ is in
$J(x_i)$ for all $1\leq i\leq k$. This clearly contradicts the local
property (P) of $x_i$, $1\leq i\leq k$. Thus there exists some
$g_0\in J(z)$ such that $|g_0(y)|<1$ and hence the local property
(P) of $y$ follows.
\end{proof}

The next result shows that the local property (P) of $x\in S_X$ is
equivalent to the C-approximate left-symmetry of $x$ in the local
sense.

\begin{lemma}\label{C-local-left-symmetry}
Let $X$ be a normed linear space. Then the local property {\rm (P)} holds
for $x\in S_X$ if and only if for $y\in x^\perp\cap S_X$, there
exists $\varepsilon_{x,y}\in [0,1)$ such that
$y\perp_B^{\varepsilon_{x,y}} x$.
\end{lemma}

\begin{proof}
We first prove the necessary part of the lemma. Suppose on the
contrary that there exists $y\in x^\perp\cap S_X$ such that
$y\not\perp_B^{\varepsilon} x$ for any $\varepsilon\in[0,1)$.
Clearly, if $f\in J(y)$, then $|f(x)|\leq 1$. Thus for all $f\in
J(y)$ we have $|f(x)|=1$ and consequently either $f\in J(x)$ or
$-f\in J(x)$. This contradicts that the local property (P) holds for
$x$ and thus the necessary part follows.

We now prove the sufficient part of the lemma. Let $y\in
x^{\bot}\cap S_X$. It follows from the assumption that
$y\bot_B^{\varepsilon_{x,y}}x$ and, equivalently, there exists $f\in
J(y)$ such that $|f(x)|\leq \varepsilon_{x,y}<1$, hence $\pm
f\not\in J(x)$. Thus $y\not\in {\mathscr A}(x)$ and the local
property (P) of $x$ follows.
\end{proof}

Observe that in the proof of Theorem~\ref{neighbourhood}, choosing
$\varepsilon_1=0$ instead of $\varepsilon_1>0$,  we obtain a weaker
version of that theorem. The following result is analogous to it.

\begin{lemma}\label{lemmaright BJ}
Let $X$ be a normed linear space and let $x,y\in S_X$ with
$x\perp_B^{\varepsilon} y$ for some $\varepsilon \in [0,1)$. Then
there exists $\delta\in(0,1-\varepsilon)$ such that
$x\perp_B^{\varepsilon+\delta} z$ for all $z\in B(y,\delta)\cap
S_X$.
\end{lemma}

\begin{proof}
Since $x\perp_B^{\varepsilon} y$ for some $\varepsilon \in [0,1)$,
there exists $f\in J(x)$ such that $|f(y)|\leq \varepsilon$. Now, if
we choose $\delta\in(0,1-\varepsilon)$, then for all $z\in
B(y,\delta)\cap S_X$, we have,
\begin{align*}
|f(z)|=|f(z)-f(y)+f(y)|\leq |f(y)|+|f(z)-f(y)|\leq
\varepsilon+\delta.
\end{align*}

Thus $x\perp_B^{\varepsilon+\delta} z$ for all $z\in B(y,\delta)\cap
S_X$.
\end{proof}

Lemma \ref{lemmaright BJ} says that the C-approximate orthogonality is stable with respect to the second vector (small perturbation of it does not cause loss of approximate orthogonality). However, as opposed to D-approximate orthogonality (see Theorem \ref{neighbourhood}), there is no analogous stability with respect to the first vector. Namely, as it can be observed in the following example, the implication
\begin{equation}\label{left-stability}
x\perp_B^{\varepsilon} y\quad \Rightarrow\quad \exists\,\delta\in (0,1-\varepsilon)\ \forall\, z\in B(x,\delta)\cap S_X:  z\perp_B^{\varepsilon+\delta} y
\end{equation}
need not be true.

\begin{example}
Let $X=\mathbb{R}^2$ with the  {\it maximum} norm. Let $\varepsilon\in[0,1)$ and
take $x=(1,1)$, $y=(-1,-\varepsilon)$, $y_0=(-1,0)$. Since $x\bot_B y_0$ and $\|y-y_0\|=\varepsilon$, it follows, via \eqref{Bchar2}, that $x\bot_B^{\varepsilon}y$.
Assuming that \eqref{left-stability} is true we take a suitable $\delta\in (0,1-\varepsilon)$ and set  $z=(1,1-\frac{\delta}{2})$. Then $z\in B(x,\delta)\cap S_X$ whence $z\bot_B^{\varepsilon'}y$ with $\varepsilon'=\varepsilon+\delta<1$. It would mean, again by \eqref{Bchar2}, that there exists $y'\in S_X$ such that $z\bot_B y'$ and $\|y'-y\|\leq\varepsilon'<1$. However, since $z^{\bot}\cap S_X=\{(0,1),(0,-1)\}$, we have $y'=(0,1)$ or $y'=(0,-1)$ but in both cases $\|y-y'\|\geq 1$ --- a contradiction. 
\end{example}

We now prove a complete characterization of the C-approximate
right-symmetry on any compact subset of $S_X$ for any normed linear
space $X$.

\begin{theorem}\label{C-approximate right-symmetry}
Let $X$ be a normed linear space and let $\mathcal{A}\subseteq S_X$
be a compact subset. Then any $y\in \mathcal{A}$ is C-approximately
right-symmetric on $\mathcal{A}$ if and only if the local property
{\rm (P)} on $\mathcal{A}$ holds for each $x\in \mathcal{A}$.
\end{theorem}

\begin{proof}
We first prove the necessary part. Suppose on the contrary that
$x\in \mathcal{A}$ is such that $x$ fails to have the local property
(P) on $\mathcal{A}$. This implies that there exists $y\in {x^\perp}
\cap\mathscr{A}(x)\cap\mathcal{A}$. Now, the C-approximate
right-symmetry of $y\in\mathcal{A}$ on $\mathcal{A}$  implies that
there exist $\varepsilon\in[0,1)$ and $g\in J(y)$ such that
$|g(x)|\leq \varepsilon$. This leads to a contradiction since
$y\in\mathscr{A}(x)$ implies $|g(x)|=1$.

We now prove the sufficient part.  Suppose on the contrary that
there exists $y\in \mathcal{A}$ such that $y$ is not C-approximately
right-symmetric on $\mathcal{A}$. Observe that if
$z\in ^\bot$${y}\cap \mathcal{A}$, then it follows from similar
arguments as those used in Lemma~\ref{C-local-left-symmetry} that
there exists $\varepsilon_{z,y}\in [0,1)$ such that
$y\perp_B^{\varepsilon_{z,y}}z$. Let $\varepsilon_{z,y}^*$ be the
infimum of all such $\varepsilon_{z,y}$. By the assumption $y$ is
not C-approximately right-symmetric on $\mathcal{A}$, this implies
that $\varepsilon_y=\underset{{z\in ^\bot}{{y}\cap \mathcal{A}}}\sup
\varepsilon_{z,y}^*=1$. Thus we can find $\{z_n\}\subseteq
^\bot$${y}\cap \mathcal{A}$ such that $y\perp_{B}
^{\varepsilon_n^*}z_n$ for $\varepsilon_n\nearrow 1$. Now, from the
compactness of $\mathcal{A}$ we can find a convergent subsequence of
$\{z_n\}$ which we again denote by $\{z_n\}$. Let
$z_n\longrightarrow z_0$, then by continuity of the norm and
compactness of $\mathcal{A}$, it follows that $z_0\in ^\bot$${y}
\cap \mathcal{A}$. Again, from similar arguments as those used in
Lemma~\ref{C-local-left-symmetry}, it follows that
$y\perp_B^{\varepsilon_{z_0,y}}z_0$ for some
$\varepsilon_{z_0,y}\in[0,1)$.

Now, if we choose $\delta\in(0,1-\varepsilon_{z_0,y})$, then it
follows from Lemma~\ref{lemmaright BJ} that there exists some
$m\in\mathbb{N}$ such that
$y\perp_B^{\varepsilon_{z_0,y}+\delta}z_k$ for all $k\geq m$. This
leads to a contradiction and thus the result follows.
\end{proof}

As an immediate consequence of the above theorem, we obtain the
following complete characterization of the C-approximate
right-symmetry of elements of $S_X$ for any finite-dimensional
Banach space.

\begin{corollary}
Let $X$ be a finite-dimensional Banach space. Then any $y\in S_X$ is
C-approximately right-symmetric if and only if the property {\rm (P)}
holds for $X$.
\end{corollary}

Now, we present a complete characterization of the C-approximate
symmetry of the Birkhoff-James orthogonality in finite-dimensional
polyhedral Banach spaces. In the proof of this characterization 
we will use the following result.

\begin{lemma}\label{same norming}
Let $X$ be a finite-dimensional polyhedral Banach space. Then for
any sequence $\{x_n\}\subseteq S_X$, we can find a sub-sequence
$\{x_{n_k}\}\subseteq\{x_n\}$ such that $J(x_{n_i})=J(x_{n_j})$ for
all $i,j\in\mathbb{N}$.
\end{lemma}

\begin{proof}
If infinitely many elements of $\{x_n\}$ are smooth points of $S_X$, 
then by the fact that there are only finitely many faces in $S_X$,
we can find a sub-sequence $\{x_{n_k}\}\subseteq
\{x_n\}\cap\mbox{sm}~S_X$ such that all the elements of the
sub-sequence $\{x_{n_k}\}$ lie in the interior of the same facet of
$S_X$. Thus there exists $f\in S_{X^*}$ such that $J(x_{n_k})=\{f\}$
for all $k\in\mathbb{N}$.

Without loss of generality we now assume that all elements of
$\{x_n\}$ are non-smooth points. If there exists $x_0\in
\mbox{Ext}\,B_X$ such that $x_n=x_0$ for infinitely many $n$, then
clearly the result follows. If not, then we can choose an edge of
$S_X$ such that its interior contains infinitely many $x_n$, as $S_X$ has only finitely many edges. It follows from \cite[Theorem 2.1]{SPBB} that every supporting functional of any $z\in S_X$ is a convex combination of the supporting functionals of the facets containing $z$. Thus the points lying in the interior of the same edge have identical set of support functionals and hence the result follows.
\end{proof}

\begin{theorem}\label{C-approximate left-symmetry}
Let $X$ be a finite-dimensional polyhedral Banach space. Then the
following properties are equivalent:
\begin{itemize}
\item[(a)] the Birkhoff-James orthogonality is C-approximately symmetric
in $X$;
\item[(b)] any $y\in S_X$ is $C$-approximately left-symmetric;
\item[(c)] the property {\rm (P)} holds for $X$;
\item[(d)] the local property {\rm (P)} holds for all $x\in {\rm Ext}\,B_X$.
\end{itemize}
\end{theorem}

\begin{proof}
Observe that $(a)$ trivially implies $(b)$.

Let us now prove $(b)\Rightarrow(c)$. Suppose on the contrary that
there exists $x\in S_X$ such that the local property (P) fails for
$x$. Thus there exists some $y\in x^\perp\cap\mathscr{A}(x)$. Now,
$y\in \mathscr{A}(x)$ implies that if $g\in J(y)$, then $|g(x)|=1$.
Thus there does not exist any $\varepsilon\in[0,1)$ such that
$y\perp_B^\varepsilon x$. This contradicts that $x$ is
C-approximately left-symmetric and $(c)$ follows.

$(c)\Rightarrow(d)$ is obvious and $(d)\Rightarrow(c)$
follows from Theorem~\ref{extreme points}.

We now show that $(c)\Rightarrow(b)$. Let $y\in S_X$ and $z\in
y^\perp\cap S_X$; then from Lemma~\ref{C-local-left-symmetry} it
follows that $z\perp_B^{\varepsilon_{y,z}} y$ for some
$\varepsilon_{y,z}\in[0,1)$. Let $\varepsilon_{y,z}^*$ be the
infimum of all such $\varepsilon_{y,z}$. Let
$\varepsilon_y=\underset{z\in y^\perp \cap
S_X}{\sup}\varepsilon_{y,z}^*$, then $\varepsilon_y \leq 1$. If
$\varepsilon_y=1$, then we can find $\{z_n\}\subseteq  y^\perp\cap
S_X$ such that $z_n\perp_{B} ^{\varepsilon_n^*}y$ for
$\varepsilon_n\nearrow 1$. Now, from the compactness of $S_X$ we can
find a convergent sub-sequence of $\{z_n\}$, which we again denote by
$\{z_n\}$, and let $z_n\longrightarrow z_0$. Then by continuity of the
norm it follows that $z_0\in y^\perp \cap S_X$. Using
Lemma~\ref{C-local-left-symmetry}, it follows that
$z_0\perp_B^{\varepsilon_{y,z_0}} y$ for some
$\varepsilon_{y,z_0}\in[0,1)$.  If $z_n=z_0$ for infinitely many
$n$'s, then clearly we obtain a contradiction. Thus without loss of
generality we can assume that $z_n\not=z_0$ for all
$n\in\mathbb{N}$.

Now, it follows from Lemma~\ref{same norming} that we can choose a
sub-sequence $\{z_{n_k}\}\subseteq \{z_n\}$ such that
$J(z_{n_i})=J(z_{n_j})$ for all $i,j\in\mathbb{N}$.

Let $\varepsilon_1\in[0,1)$ be such that
$z_{n_1}\perp_B^{\varepsilon_1} y$. Then the choice of the
sub-sequence $\{z_{n_k}\}$ will ensure that
$z_{n_k}\perp_B^{\varepsilon_1} y$ for all $k\geq 1$. This leads to
a contradiction. Thus $\varepsilon_y < 1$, $z\perp_B^{\varepsilon_y}
y$ and hence $y$ is C-approximately left-symmetric.

Now, we show that $(b)\Rightarrow(a)$. Suppose on the contrary that
the Birkhoff-James orthogonality is not C-approximately symmetric in
$X$. From the equivalence of $(b)$ and $(c)$, it follows that each
$x\in S_X$ has the local property (P). If $x,y\in S_X$ are such that
$x\perp_B y$, then it follows by Lemma~\ref{C-local-left-symmetry}
that there exists $\varepsilon_{x,y}\in[0,1)$ such that
$y\perp_B^{\varepsilon_{x,y}}x$. Let $\{x_n\}$, $\{y_n\}\subseteq
S_X$ be such that $x_n\perp_B y_n$,
$y_n\perp_{B}^{\varepsilon_n^*}x_n$ for $\varepsilon_n\nearrow 1$.
Since $S_X$ is compact, it follows that there exist convergent
sub-sequences of $\{x_n\}$, $\{y_n\}$, which we again denote by
$\{x_n\}$, $\{y_n\}$, respectively. Let $x_0,y_0\in S_X$ be such
that $x_n\longrightarrow x_0$ and $y_n\longrightarrow y_0$. Now,
from the continuity of the norm it follows that $x_0\perp_B y_0$.
Using Lemma~\ref{C-local-left-symmetry}, we can find
$\varepsilon_{x_0,y_0}\in[0,1)$ such that
$y_0\perp_B^{\varepsilon_{x_0,y_0}}x_0$.

We now prove that we can choose $\{x_n\}$ such that $x_n\not=x_0$
for almost all $n\in\mathbb{N}$. If there exists a sub-sequence
$\{x_{n_k}\}\subseteq \{x_n\}$ such that $x_{n_k}=x_0$ for all
$k\in\mathbb{N}$, then $x_0\perp_B y_{n_k}$ for all $k\in\mathbb{N}$.
Then by taking $\mathcal{A}=S_X$ in Theorem~\ref{C-approximate
right-symmetry}, it follows that
$y_{n_k}\perp_B^{\varepsilon_{x_0}}x_0$ for some
$\varepsilon_{x_0}\in[0,1)$ and for all $k\in\mathbb{N}$. Thus
$y_{n_k}\perp_B^{\varepsilon_{x_0}}x_{n_k}$ for some
$\varepsilon_{x_0}\in[0,1)$ and for all $k\in\mathbb{N}$. Clearly,
this contradicts that $y_n\perp_{B}^{\varepsilon_n^*}x_n$ for
$\varepsilon_n\nearrow 1$. Thus we can assume that $x_n\not=x_0$ for
almost all $n\in\mathbb{N}$. Also, by using the similar arguments we
can assume that $y_n\not=y_0$ for almost all $n\in\mathbb{N}$.

It follows from Lemma~\ref{same norming} that we can find
sub-sequences $\{x_{n_k}\}$, $\{y_{n_k}\}$ of $\{x_n\}$, $\{y_n\}$,
respectively, such that $J(x_{n_i})=J(x_{n_j})$ and
$J(y_{n_i})=J(y_{n_j})$ for all $i,j\in\mathbb{N}$. Observe that
$x_{n_k}\perp_B y_{n_k}$ for all $k\in\mathbb{N}$ and each $x\in
S_X$ has local property (P) which implies the following:
\begin{itemize}
\item[(i)] if $\{x_{n_k}\}$, $\{y_{n_k}\}\subseteq S_X\setminus (\mbox{sm}~S_X\cup \mbox{Ext}\,B_X)$, then
elements of $\{x_{n_k}\}$ and $\{y_{n_k}\}$ lie in  the interiors of
different edges of $S_X$.
\item[(ii)] for all other cases elements of $\{x_{n_k}\}$,
$\{y_{n_k}\}$ lie on different facets of $S_X$.
\end{itemize}

Now, from the choice of the sub-sequences $\{x_{n_k}\}$ and
$\{y_{n_k}\}$ it follows that
\begin{align*}
x_{n_i}\perp_B y_{n_j}~~\mbox{for~all~}i,j\in\mathbb{N}.
\end{align*}

Thus $x_0\perp_B y_{n_j}$ for all $j\in\mathbb{N}$. Using $(b)$ we
can find $\varepsilon_{x_0}\in[0,1)$ such that
$y_{n_j}\perp_B^{\varepsilon_{x_0}} x_0$ for all $j\in\mathbb{N}$.

Let $f\in J(y_{n_1})$ be such that $|f(x_0)|\leq \varepsilon_{x_0}$.
Now, if we choose $\delta\in(0,1-\varepsilon_{x_0})$, then for  $z\in
B(x_0,\delta)\cap S_X$, we have,
\begin{align*}
|f(z)|\leq |f(z)-f(x_0)|+|f(x_0)|\leq \delta+\varepsilon_{x_0}.
\end{align*}

The choice of $\{y_{n_k}\}$ implies that $f\in J(y_{n_j})$ for all
$j\in\mathbb{N}$. Thus we can find $m\in\mathbb{N}$ such that
\begin{align*}
y_{n_j}\perp_B^{\varepsilon_{x_0}+\delta}
x_{n_j}~~\mbox{for~all~}~j\geq m.
\end{align*}

This clearly contradicts that $x_n\perp_{B}^{\varepsilon_n^*}y_n$
for $\varepsilon_n\nearrow 1$. Thus the Birkhoff-James orthogonality
is C-approximately symmetric in $X$.
\end{proof}

In view of Theorem~\ref{C-approximate right-symmetry} and
Theorem~\ref{C-approximate left-symmetry} we would like to propose
the following conjecture.

\begin{conjecture}\label{symmetry BJ}
Let $X$ be a finite-dimensional Banach space. Then the
Birkhoff-James orthogonality is C-approximately symmetric in $X$ if
and only if the local property (P) holds for each $x\in\mbox{Ext}\,B_X$.
\end{conjecture}

The above conjecture is not true in general. In \cite[Example
3.7]{CW1}, an infinite-dimensional smooth space (hence the local
property (P) holds for each $x\in S_X$), which is not
C-approximately symmetric was constructed.

\bigskip

In \cite[Theorem 4.2]{CW1}, Chmieli\'{n}ski and W\'{o}jcik proved
that in finite-dimensional smooth Banach spaces, the Birkhoff-James
orthogonality is C-approximately symmetric. We now generalize the
result by proving that in uniformly smooth Banach spaces, the
Birkhoff-James orthogonality is C-approximately symmetric on any
compact subset of $S_X$.
\begin{theorem}
Let $X$ be a uniformly smooth Banach space and let
$\mathcal{A}\subseteq S_X$ be a compact subset. Then the
Birkhoff-James orthogonality is C-approximately symmetric on
$\mathcal{A}$.
\end{theorem}

\begin{proof}
For $x\in S_X$ let $J(x)=\{f_x\}$. First we claim that if
$\{x_n\}\subseteq S_X$ and $x\in S_X$ such that $x_n\longrightarrow
x$, then $f_{x_n}\longrightarrow f_x$. Suppose on the contrary that
$x_n\longrightarrow x$ and $f_{x_n}\not\longrightarrow f_x$. Then
for given $\varepsilon>0$ there exists $\{x_{n_k}\}\subseteq\{x_n\}$
such that $\|f_{x_{n_k}}-f_x\|>\varepsilon$ for all
$k\in\mathbb{N}$. Clearly, $\varepsilon<2$. Since $X$ is uniformly
smooth, $X^*$ is uniformly convex. Thus there exists
$\delta(\varepsilon)>0$ such that
$\|f_{x_{n_k}}+f_x\|<2-\delta(\varepsilon)$ for all
$k\in\mathbb{N}$. Also, for all $k\in\mathbb{N}$, we have,
\begin{align*}
\|f_{x_{n_k}}+f_x\|\geq
f_{x_{n_k}}(x_{n_k})+f_x(x_{n_k})=1+f_x(x_{n_k}).
\end{align*}
As $x_{n_k}\longrightarrow x$, we get,
\begin{align*}
2\leq \lim_{k\longrightarrow\infty}\|f_{x_{n_k}}+f_x\|\leq
2-\delta(\varepsilon).
\end{align*}
This leads to a contradiction and thus our claim follows.

It follows from the smoothness of $X$ that any $x\in S_X$ has local
property (P). If $x$, $y\in S_X$ are such that $x\perp_B y$, then by
using Lemma~\ref{C-local-left-symmetry}, we can find
$\varepsilon_{x,y}\in[0,1)$ such that $y\perp_B^{\varepsilon_{x,y}}
x$.

Now, we will prove that the Birkhoff-James orthogonality is
C-approximately symmetric on any compact subset
$\mathcal{A}\subseteq S_X$. Suppose on the contrary that the
Birkhoff-James orthogonality is not C-approximately symmetric on
some compact subset $\mathcal{A}\subseteq S_X$. Thus we can find
$\{x_n\}$, $\{y_n\}\subseteq \mathcal{A}$ such that $x_n\perp_B y_n$
and $y_n\perp_{B}^{\varepsilon_n^*} x_n$ for some
$\varepsilon_n\nearrow 1$.

From compactness of $\mathcal{A}$ we can find convergent
sub-sequences of $\{x_n\}$, $\{y_n\}$ which we again denote by
$\{x_n\}$ and $\{y_n\}$, respectively. Let $x_0$, $y_0\in
\mathcal{A}$ be such that $x_n\longrightarrow x_0$ and
$y_n\longrightarrow y_0$. Using continuity of the norm it follows
that $x_0\perp_B y_0$ and thus there exists $\varepsilon\in[0,1)$
such that $y_0\perp_B^\varepsilon x_0$.

Let $\varepsilon_1\in(0,1-\varepsilon)$. Then there exists
$m_1\in\mathbb{N}$ such that $\|f_{y_n}-f_{y_0 }\|\leq
\varepsilon_1$ for all $n\geq m_1$. Thus for all $n\geq m_1$, we
have,
\begin{align*}
|f_{y_n}(x_0)|\leq |f_{y_n}(x_0)-f_{y_0}(x_0)|+|f_{y_0}(x_0)|\leq
\varepsilon_1+\varepsilon
\end{align*}
and this implies
\begin{align*}
y_n\perp_B^{\varepsilon+\varepsilon_1}x_0.
\end{align*}
Now, if we choose $\delta\in (0,1-\varepsilon-\varepsilon_1)$, then
we can find $m_2\in\mathbb{N}$ such that $x_n\in B(x_0,\delta)$ for
all $n\geq m_2$.

Let $m=\max\{m_1,m_2\}$. Then for all $z\in B(x_0,\delta)$ and for
all $n\geq m$, we have,
\begin{align*}
|f_{y_n}(z)|\leq |f_{y_n}(z)-f_{y_n}(x_0)|+|f_{y_n}(x_0)|\leq
\|z-x_0\|+|f_{y_n}(x_0)|<\delta+\varepsilon+\varepsilon_1.
\end{align*}

Thus
\begin{align*}
y_n\perp_B^{\delta+\varepsilon+\varepsilon_1}x_n~~\mbox{for~all}~n\geq
m.
\end{align*}
This clearly leads to a contradiction and thus the result follows.
\end{proof}

\bigskip

For a normed linear space $X$, the following constant was defined in
\cite{SW}:
\begin{align*}
\mathcal{R}(X):=\sup\{\|x-y\|:\overline{xy}\subset S_X\}.
\end{align*}

\begin{remark}\label{remark RX}
If $X$ is a finite-dimensional polyhedral Banach space and
$\mathcal{R}(X)\leq 1$, then the local property (P) holds for each
$x\in \mbox{Ext}\,B_X$. To see this observe that if the local
property (P) fails for $x\in \mbox{Ext}\,B_X$, then there exists $y\in
S_X$ such that $x\perp_B y$ and if $f\in J(y)$, then either $f\in
J(x)$ or $-f\in J(x)$. This implies that one of the following holds
true:
\begin{itemize}
\item[(i)] $y$ lies in the interior of one of the associated edges
of $x$ or $-x$;
\item[(ii)] $y$ lies in the interior of one of the associated facets
of $x$ or $-x$.
\end{itemize}

Thus either $\overline{xy}\subset S_X$ or $\overline{-xy}\subset
S_X$. If $\overline{xy}\subset S_X$, then $x\perp_B y$ implies

\begin{align*}
\mathcal{R}(X)>\|x-y\|\geq\|x\|=1.
\end{align*}

Now, we assume that $\overline{-xy}\subset S_X$.  By the homogeneity
of orthogonality and $x\perp_B y$ we get $-x\perp_B y$. Thus,
\begin{align*}
\mathcal{R}(X)>\|(-x)-y\|\geq\|-x\|=1.
\end{align*}
\end{remark}

The following example shows that the converse is not true, that is,
a two-dimensional polyhedral Banach space satisfying the local
property (P) for each $x\in \mbox{Ext}\,B_X$ need not necessarily
satisfy $\mathcal{R}(X)\leq 1$. Consider a two-dimensional
polyhedral Banach space $X=\mathbb{R}^2$, whose unit sphere is
determined by the extreme points $v_1=(2,2)$, $v_2=(1,3)$,
$v_3=(0,3.5)$, $v_4=(-1,3)$, $v_5=(-2,2)$, $v_6=-v_1$, $v_7=-v_2$,
$v_8=-v_3$, $v_9=-v_4$, $v_{10}=-v_5$. For this space
$\mathcal{R}(X)>1$, Figures 3 and 4, given below show that the local
property (P) holds for each $x\in \mbox{Ext}\,B_X$. In Figure 3, $f$
is the supporting functional corresponding to the edge
$\overline{v_1v_{10}}$, $g$ is the supporting functional
corresponding to the edge $\overline{v_1v_2}$ and $h$ is the
supporting functional corresponding to the edge $\overline{v_2v_3}$.
In Figure 4, $h$ is the supporting functional corresponding to the
edge $\overline{v_3v_4}$, $g$ is the supporting functional
corresponding to the edge $\overline{v_4v_5}$ and $f$ is the
supporting functional corresponding to the edge $\overline{v_5v_6}$.

\begin{figure}
     \centering
     \begin{subfigure}{.50\textwidth}
         \begin{tikzpicture}[scale=0.88]
\draw (2,2) -- (1,3); \draw (0,3.5) --
(1,3); \draw (0,3.5) -- (-1,3); \draw (-2,2) -- (-1,3); \draw (-2,2)
-- (-2,-2); \draw (-2,-2) -- (-1,-3); \draw (0,-3.5) -- (1,-3);
\draw (2,-2) -- (1,-3); \draw (2,2) -- (2,-2); \draw (-1,-3) --
(0,-3.5); \draw [dashed](3.5,0) -- (-3.5,0);

\draw[dashed,blue] (0,5) -- (0,-5); \draw[dashed,blue] (2,4.5) --
(2,-4.5);

\draw[dashed,red] (-3,3) -- (3,-3); \draw[dashed,red] (4,0) --
(-1,5);

\draw[dashed,teal] (4,-2) -- (-4,2); \draw[dashed,teal] (4,1.5) --
(-2,4.5);

\draw [fill] (2,2) circle [radius=.08]; \draw [fill] (2,-2) circle
[radius=.08]; \draw [fill] (-2,2) circle [radius=.08]; \draw [fill]
(-2,-2) circle [radius=.08]; \draw [fill] (0,3.5) circle
[radius=.08]; \draw [fill] (0,-3.5) circle [radius=.08]; \draw
[fill] (1,3) circle [radius=.08]; \draw [fill] (-1,3) circle
[radius=.08]; \draw [fill] (-1,-3) circle [radius=.08]; \draw [fill]
(1,-3) circle [radius=.08];

\draw (2.4,2)  node {$v_1$}; \draw (2.4,-2)  node {$v_{10}$}; \draw
(-2.4,2)  node {$v_5$}; \draw (-2.4,-2)  node {$v_6$}; \draw (1.4,3)
node {$v_2$}; \draw (-1.4,3)  node {$v_4$}; \draw (-1.4,-3)  node
{$v_7$}; \draw (1.4,-3)  node {$v_9$}; \draw (-0.4,3.5)  node
{$v_3$}; \draw (-0.4,-3.5)  node {$v_8$};

\draw (2.4,-4) node{$f$}; \draw (0.7,-4) node{$\mbox{ker}~ f$};
\draw (-2.3,3.2) node{$\mbox{ker} ~g$}; \draw (-1,4.6) node{$g$};
\draw (-2,4.2) node{$h$}; \draw (-4.1,1.5) node{$\mbox{ker}~ h$};

\draw(0,-6) node{Figure 3};
\end{tikzpicture}

     \end{subfigure}
       \begin{subfigure}{.50\textwidth}
         \centering
         \begin{tikzpicture}[scale=0.88]
\draw (2,2) -- (1,3); \draw (0,3.5) -- (1,3); \draw (0,3.5) --
(-1,3); \draw (-2,2) -- (-1,3); \draw (-2,2) -- (-2,-2); \draw
(-2,-2) -- (-1,-3); \draw (0,-3.5) -- (1,-3); \draw (2,-2) --
(1,-3); \draw (2,2) -- (2,-2); \draw (-1,-3) -- (0,-3.5);

\draw[dashed,blue] (0,5) -- (0,-5); \draw[dashed,blue] (-2,4.5) --
(-2,-4.5);

\draw[dashed,red] (-4,0) -- (1,5); \draw[dashed,red] (3,3) --
(-3,-3);

\draw[dashed,teal] (-4,-2) -- (3.5,1.75); \draw[dashed,teal] (-4,1.5) --
(2,4.5); \draw[dashed] (3.5,0) -- (-3.5,0);

\draw [fill] (2,2) circle [radius=.08]; \draw [fill] (2,-2) circle
[radius=.08]; \draw [fill] (-2,2) circle [radius=.08]; \draw [fill]
(-2,-2) circle [radius=.08]; \draw [fill] (0,3.5) circle
[radius=.08]; \draw [fill] (0,-3.5) circle [radius=.08]; \draw
[fill] (1,3) circle [radius=.08]; \draw [fill] (-1,3) circle
[radius=.08]; \draw [fill] (-1,-3) circle [radius=.08]; \draw [fill]
(1,-3) circle [radius=.08];

\draw (2.4,2)  node {$v_1$}; \draw (2.4,-2)  node {$v_{10}$}; \draw
(-2.4,2)  node {$v_5$}; \draw (-2.4,-2)  node {$v_6$}; \draw (1.4,3)
node {$v_2$}; \draw (-1.4,3)  node {$v_4$}; \draw (-1.4,-3)  node
{$v_7$}; \draw (1.4,-3)  node {$v_9$}; \draw (-0.4,3.5)  node
{$v_3$}; \draw (-0.4,-3.5)  node {$v_8$};

\draw (0.7,-4) node{$\mbox{ker}~ f$}; \draw (-2.4,-4) node{$f$};
\draw (2.3,3) node{$\mbox{ker}~ g$}; \draw (0.5,4.8) node{$g$};
\draw (-3.5,-1.2) node{$\mbox{ker}~ h$}; \draw (-3.7,1.95) node{$h$};

\draw(0,-6) node{Figure 4};
\end{tikzpicture}
     \end{subfigure}

\end{figure}

As a consequence of Theorem~\ref{C-approximate left-symmetry} and
Remark~\ref{remark RX} we obtain the next result.

\begin{theorem}\label{RX}
Let $X$ be a finite-dimensional polyhedral Banach space such that $R(X)\leq 1$. Then the Birkhoff-James orthogonality is C-approximately symmetric.
\end{theorem}

\begin{remark}
The previous example shows that the converse of the above
theorem is not true. It follows also from \cite[Corollary 3.9]{CW1} that for any finite-dimensional Banach space $X$ (not necessarily polyhedral) if $\mathcal{R}(X)<1$, then the Birkhoff-James orthogonality is C-approximately symmetric. However, polyhedralness is essential for $\mathcal{R}(X)=1$. Indeed, consider the $l_2-l_{\infty}$ norm on the plane (see Figure 5) and vectors $x=(1,1)$, $y=(0,1)$. Then $x\bot_B y$ but $y\not\hspace{-0.3em}\bot_B^{\varepsilon} x$ for any $\varepsilon\in [0,1)$.
 
\begin{figure}[!htb]
\centering
\begin{tikzpicture}[scale=2]
\draw[dashed] (-2,0)--(2,0);
\draw[dashed] (0,-1.5)--(0,1.5);

\draw (-1,0) arc(180:90:1) (0,1);
\draw (0,-1) arc(270:360:1) ;

\draw (0,1)--(1,1)--(1,0);
\draw (-1,0)--(-1,-1)--(0,-1);
(1,0)
\draw[thick,->](0,0)--(1,1);
\draw[thick,->](0,0)--(0,1);

\node[right] at (0.5,0.5) {{ $x$}};
\node[left] at (0,0.5) {{ $y$}};

\node[below right] at (1,0) {{$1$}};
\node[below left] at (-1,0) {{$-1$}};
\node[above right] at (0,1) {{$1$}};
\node[below right] at (0,-1) {{ $-1$}};
\node at (0,-1.7) {Figure 5};

\end{tikzpicture}
\end{figure}
\end{remark}

\pagebreak

\section{C-approximate symmetry for two-dimensional polyhedral Banach spaces}

We introduce yet another property of a normed linear space $X$:
\begin{align*}
\mbox{if}~ \overline{xy}\subset S_X~\mbox{and}~ x\perp_B y,
~\mbox{then}~ x,y\in \mbox{Ext}\,B_X.\hspace{1cm} (\mbox{P1})
\end{align*}

We first prove that in any finite-dimensional polyhedral Banach space property (P1) always implies the local property (P) for each $x\in S_X$.

\begin{lemma}
Let $X$ be a finite-dimensional polyhedral Banach space such that property {\rm (P1)} holds for $X$. Then the local property {\rm (P)} holds for each $x\in S_X$.
 \end{lemma}
 
\begin{proof}
It follows from Theorem~\ref{extreme points} that it is sufficient to prove that the local property (P) holds for each $x\in \mbox{Ext}\,B_X$. Suppose on the contrary that there exists $x\in \mbox{Ext}\,B_X$ such that the local property (P) fails for $x$. Thus there exists $y\in x^\perp\cap S_X$ such that if $f\in J(y)$, then either $f\in J(x)$ or $-f\in J(x)$. This clearly shows that either $\overline{xy}\subset S_X$ or $\overline{-xy}\subset S_X$. Now, property (P1) of $X$ shows that $y\in \mbox{Ext}\,B_X$. This contradicts that the local property (P) fails for $x$ and thus the result follows.
\end{proof}

We will see that the converse is also true in two-dimensional spaces but in general it is not true.
To exhibit this we now give an example of a three-dimensional
polyhedral Banach space $X$, for which the local property (P) holds for all elements of $S_X$ but $X$ fails to have the property (P1).

 \pgfmathsetmacro\angFuite{155}
 \pgfmathsetmacro\coeffReduc{1}
 \pgfmathsinandcos\sint\cost{\angFuite}

\begin{tikzpicture}[current plane/.estyle=%
        {cm={0.9,0,\coeffReduc*\cost,-\coeffReduc*\sint,(0,#1)}},x=0.5cm,y=0.5cm,z=0.3cm][scale=0.1]

\GraphInit[vstyle=Simple]
\begin{scope}[current plane=-3cm]
\SetGraphShadeColor{black!70}{white}{white}
\grEmptyCycle[Math,RA=3,prefix=a]{6}
\end{scope}

\begin{scope}[current plane=0 cm]
\SetGraphShadeColor{white}{white}{white}
\grEmptyCycle[Math,RA=6,prefix=b]{6}
\end{scope}
\begin{scope}[current plane=3 cm]
\SetGraphShadeColor{black!70}{white}{white}
\grEmptyCycle[Math,RA=3,prefix=c]{6}
\end{scope}
\begin{scope}[current plane=-5.5 cm]
\node(0,0,0) {Figure 6};
\end{scope}
\SetGraphShadeColor{white}{white}{black} \EdgeInGraphSeq{a}{0}{1}
\EdgeInGraphSeq{b}{0}{1} \EdgeInGraphSeq{c}{0}{1}
\EdgeIdentity*{a}{b}{0,1} \Edge(a0)(a5) \Edge(b0)(b5) \Edge(c0)(c5)
\Edge(b2)(a2) \Edge(c2)(b2)

\Edge(b5)(a5) \Edge(c5)(b5)
\Edge[style={opacity=.5}](b3)(a3)

\Edge[style={opacity=.5}](c3)(b3)
\Edge[style={opacity=.3}](b4)(a4)

\Edge[style={opacity=.3}](c4)(b4)

\Edge[style={opacity=.3}](b4)(b3)
\Edge[style={opacity=.3}](c4)(c3)
\Edge[style={opacity=.3}](a2)(a3)
\Edge[style={opacity=.3}](b2)(b3)
\Edge[style={opacity=.3}](c2)(c3)
\Edge[style={opacity=.3}](a4)(a3)
\Edge[style={opacity=.3}](b4)(b3)
\Edge[style={opacity=.3}](c4)(c3)
\Edge[style={opacity=.3}](a4)(a5)

\Edge[style={opacity=.3}](b4)(b5)
 \Edge[style={opacity=.3}](c4)(c5)
\EdgeIdentity*{c}{b}{0,1} \draw[dashed,->] (xyz cs:x=-11) -- (xyz
cs:x=11) node[above] {}; \draw[dashed, ->] (xyz cs:y=-10) -- (xyz
cs:y=10) node[right] {}; \draw[dashed,->] (xyz cs:z=-14) -- (xyz
cs:z=14) node[above] {};
\end{tikzpicture}

\begin{example}

Consider a three-dimensional polyhedral Banach space
$X=\mathbb{R}^3$, whose unit sphere is given in Figure 6.
Observe that the local property (P) holds for each $x\in S_X$. If we
consider any extreme point of $B_X$ which is black in color, then it
is orthogonal to all the elements which lie in the intersection of
$S_X $ and the plane passing through all the extreme points of $B_X$
which are white in color. This shows that $X$ fails to have property
(P1).

\end{example}

We now show that for any two-dimensional polyhedral Banach spaces the property (P1) and the local property (P) for each $x\in S_X$ are equivalent. The next result provides a complete characterization of the C-approximate symmetry of the Birkhoff-James orthogonality in a two-dimensional polyhedral Banach space.

\begin{theorem}\label{two dimensional polygonal equivalence}
Let $X$ be a two-dimensional polyhedral Banach space. Then the
following properties are equivalent:
\begin{itemize}
\item[(a)] the property {\rm (P1)} holds for $X$;
\item[(b)] if $x\in {\rm Ext}\,B_X$ and $y\in S_X$ are such that $x\perp_B y$, then $y$ does not lie
in the interior of any of the associated edges of $\pm x$;
\item[(c)] the local property {\rm (P)} holds for each $x\in S_X$;
\item[(d)] the Birkhoff-James orthogonality is C-approximately symmetric in
$X$.
\end{itemize}
\end{theorem}

\begin{proof}
Equivalence of $(c)$ and $(d)$ follows from
Theorem~\ref{C-approximate left-symmetry}. To complete the proof we
now prove the equivalence of $(a)$, $(b)$ and $(c)$.

We first prove $(a)\Rightarrow(b)$. Let $x\in\mbox{Ext}\,B_X$ and
$y\in x^\perp\cap S_X$. If $y$ does not lie on the associated edges
of $\pm x$, then the result follows trivially. Suppose that $y$ lies on
one of the associated edges of $x$ or $-x$. Then either
$\overline{xy}\subseteq S_X$ or $\overline{-xy}\subseteq S_X$. Now,
it follows from $(a)$ that $y\in \mbox{Ext}\,B_X$ and thus $(b)$
follows.

We now show $(b)\Rightarrow (c)$. Observe that to show $(c)$, it is
sufficient to show that the local property (P) holds for any
$x\in\mbox{Ext}\,B_X$. Let $x\in\mbox{Ext}\,B_X$ and let $y\in
x^\perp\cap S_X$. It follows from $(b)$ that $y$ does not lie in the
interior of any of the associated edges of $\pm x$. Thus there
exists some $f\in J(y)$ such that $\pm f\not\in J(x)$ and hence
$(c)$ follows.

Now, we prove that $(c)\Rightarrow (a)$. Let $x\in S_X$ and let
$y\in x^\perp\cap S_X$ be such that $\overline{xy}\subseteq S_X$
which implies that $x\in \mbox{Ext}\,B_X$. The local property (P) of
$x$ implies that $y$ cannot be an interior point of any of the
associated edges of $x$. Thus $y\in \mbox{Ext}\,B_X$ and $(a)$
follows. This completes the proof.
\end{proof}

\pagebreak

Applying Theorem~\ref{RX}, we now prove that in any two-dimensional regular
polyhedral Banach space with at least $6$ vertices, the Birkhoff-James orthogonality is C -approximately symmetric. Regularity here means that all edges of the unit sphere are of the same length with respect to the Euclidean metric and all the interior angles are of the same measure. 
\begin{theorem}\label{2n polygon}
Let $X$ be a two-dimensional regular polyhedral space with $2n$
vertices, $n\in \mathbb{N}$, $n\geq 3$. Then the Birkhoff-James
orthogonality is C-approximately symmetric in $X$.
\end{theorem}

\begin{proof}

Without loss of generality we assume that all the vertices of the polygon lie on the Euclidean unit sphere and they are: 
$$v_j=\left(\cos\frac{(2j-1)\pi}{2n},\sin\frac{(2j-1)\pi}{2n}\right),\qquad 
j=1,\ldots,2n
$$ 
(for $n=3$ and $n=4$ the situation is illustrated in Figures 7
and 8).

Let the ordinate meet the boundary of the polygon at $\pm(0,\beta)$. Then $\beta=1$ if $n$ is odd and $\beta=\cos{\frac{\pi}{2n}}\geq\cos\frac{\pi}{8}>0.9$ if $n$ is even.

Let $L=\overline{v_1v_{2n}}$. The value of $\mathcal{R}(X)$ is equal to the length of $L$. 
To determine the latter one we will translate $L$ by the vector $-v_{2n}$, which means that $v_{2n}$ is moved to the origin $o$ and $v_1$ to some point $u$ on the ordinate. 

The Euclidean length of $L$ is equal to $|L|=2\sin\frac{\pi}{2n}$. Then $|L|=1$ for $n=3$ and $|L|\leq 2\sin\frac{\pi}{8}<0.8<\beta$ for $n\geq 4$. Therefore $u$ lies on the boundary of the polygon for $n=3$ and inside it for $n\geq 4$, hence the length of $L$ (with respect to the introduced norm) is not greater than 1. 
Thus we have $\mathcal{R}(X)\leq 1$ and the assertion follows now from Theorem~\ref{RX}.

\bigskip

\hspace{1cm}
\begin{tikzpicture}[scale=0.9]
\newdimen\R
\R=2cm
\hspace{-0.3cm}
\draw (30:\R) \foreach \x in {30,90,150,210,270,330} {  --
(\x:\R) }
 -- cycle (330:\R) 
 -- cycle (270:\R) 
      -- cycle (210:\R) node[below] {$v_4$}
 -- cycle (150:\R) 
-- cycle  (90:\R)  -- cycle  (90:\R)  -- cycle  (30:\R);

\draw [fill] (0,2) circle [radius=.07]; \draw [fill] (0,-2)
circle [radius=.07]; \draw [fill] (0,2) circle [radius=.08];

\draw [fill] (0,0) circle [radius=.07];

\draw [fill] (1.72,1) circle [radius=.07]; \draw [fill]
(1.72,-1) circle [radius=.07]; \draw [fill] (-1.72,1) circle
[radius=.07]; \draw [fill] (-1.71,-1) circle [radius=.07];

\coordinate (A) at (0,2); \coordinate (B) at (1.73,1);  

\coordinate (D) at (0,-2); \draw [add= 0.35 and 0.35,black, dashed]
(A) to (D); \coordinate (E) at (-1.32,.75);  \coordinate (F) at
(1.32,-.75); 

\coordinate (G) at (-2,0); \coordinate (H) at (2,0); \draw
[add= .4 and .4, black,dashed] (G) to (H);

\draw (2.5,1) node{$v_1$};
\draw (-1.8,1) node{$v_3$};

\draw (2.5,-1) node{$v_6$};

\draw (0.8,2) node{$v_2$};
\draw (0.8,-2) node{$v_5$};
\draw (0,2) node{$u$};
\draw (2.25,.3)node {$L$};
\draw (-0.1,-.3)node {$o$};

\draw (0,-4) node {Figure 7};



\draw (9.35,.74)--(9.35,-.74) node{};

\draw (5.65,.74)--(5.65,-.74) node{};

\draw (8.24,1.85)--(9.35,.74) node{};

\draw (8.24,-1.85)--(6.76,-1.85) node{};

\draw (8.24,1.85)--(6.76,1.85) node{};

\draw (5.65,.74)--(6.76,1.85) node{};

\draw (5.65,-.74)--(6.76,-1.85) node{};

\draw (9.35,-.74)--(8.24,-1.85) node{};

\draw [fill] (9.35,.74) circle [radius=.07];

\draw [fill] (5.65,.74) circle [radius=.07];

\draw [fill] (5.65,-.74) circle [radius=.07];

\draw [fill] (9.35,-.74) circle [radius=.07];

\draw [fill] (8.24,1.85) circle [radius=.07];

\draw [fill] (6.76,-1.85) circle [radius=.07];

\draw [fill] (6.76,1.85) circle [radius=.07];

\draw [fill] (8.24,-1.85) circle [radius=.07];

\draw [fill] (7.5,0) circle [radius=.07];

\draw [fill] (7.5,1.48) circle [radius=.07];

\coordinate (A) at (9.35,.74); \coordinate (B) at (9.35,-.74); 

\coordinate (C) at (8.24,1.85);  

\coordinate (D) at (7.5,.965);  \coordinate (E) at (7.5,-.965);

\draw[add= 1.3 and 1.3, black, dashed] (D) to (E);

\coordinate (F) at (6.535,.965); \coordinate (G) at (8.465,-.965);

\coordinate (H) at (9.5,0); \coordinate (I) at (5.5,0); \draw
[add=0.38 and 0.4, black, dashed] (H) to (I);

\draw [fill] (7.5,1.84) circle [radius=.07];

\draw [fill] (7.5,-1.84) circle [radius=.07];
\draw (10,.74) node{$v_1$};
\draw (10,-.74) node{$v_8$};

\draw (5.4,.74) node{$v_4$};

\draw (5.4,-.74) node{$v_5$}; 
\draw (6.5,1.85) node{$v_3$};
\draw (9,1.9) node{$v_2$};
\draw (8.9,-1.85) node{$v_7$};
\draw (6.5,-1.85) node{$v_6$};
\draw (9.95,.3)node {$L$};
\draw (7.4,-.3)node {$o$};
\draw (7.4,1.48)node {$u$};
\draw (8,2.2)node {\tiny$(0,\beta)$};
\draw (8,-2.2)node {\tiny$(0,-\beta)$};
\draw (7.5,-4) node {Figure 8};
\end{tikzpicture}
\end{proof}

The following example shows that in Theorem~\ref{2n polygon}, the
regularity condition cannot be avoided. Consider a two-dimensional
polyhedral Banach space $X=\mathbb{R}^2$, whose unit sphere is given
on Figure 9 below. 

\begin{center}

\begin{tikzpicture}[scale=0.9]
\draw (2,0.5) -- (0,2);
\draw (0,2) -- (-2,0.5);
\draw (-2,-.5) -- (-2,0.5);
\draw (-2,-.5) -- (0,-2);
\draw (0,-2) -- (2,-0.5);
\draw(2,-.5)--(2,0.5);

\draw[dashed,blue] (-2,3.5) -- (4,-1);
\draw[dashed,blue] (-3,2.25) -- (3,-2.25);

\draw[dashed,red] (2,3.5) -- (-4,-1);

\draw [fill] (2,0.5) circle [radius=.08];
\draw [fill] (0,-2) circle [radius=.08];
\draw [fill] (0,2) circle [radius=.08];
\draw [fill] (-2,0.5) circle [radius=.08];
\draw [fill] (-2,-.5) circle [radius=.08];
\draw [fill] (2,-.5) circle [radius=.08];
\draw [fill] (-1.333333333,1) circle [radius=.08];
\draw [fill] (1.333333333,-1) circle [radius=.08];

\draw (2.4,0.5)  node {$v_1$};
\draw (-2.4,0.5)  node {$v_3$};
\draw (-2.4,-0.5)  node {$v_4$};
\draw (2.4,-0.5)  node {$v_6$};
\draw (0.5,2)  node {$v_2$};
\draw (0.5,-2)  node {$v_5$};

\draw [dashed](0,-3.5)--(0,3.5);
\draw[dashed](-3.5,0)--(3.5,0);

\draw(-1.333,1.3) node{$y$}; \draw(1.95,-1.1) node{$-y$};
\draw(4,-0.5) node{$f$}; \draw(3.4,-2) node{$\mbox{ker}~f$};
\draw(-4,-0.5) node{$g$};

\draw(0,-4) node{Figure~9};
\end{tikzpicture}
\end{center}
Clearly, $y\in v_2^\perp\cap S_X$. Also, $y$ lies
in the interior of the edge $\overline{v_2v_3}$. The only supporting
linear functional for $y$ is the supporting functional $g\in
S_{X^*}$ corresponding to the edge $\overline{v_2v_3}$ such that
$g(x)=1$ for all $x\in\overline{v_2v_3}$. Thus $v_2$ is not
C-approximately left-symmetric.

\section{Acknowledgments}
A joint work on this article began in January 2020, during the first author's visit
to the Jadavpur University in Kolkata. The
hospitality of Professor Kallol Paul is acknowledged with gratitude.
Dr. Debmalya Sain feels elated to acknowledge the motivating
presence of his beloved brother Krittish Roy in his life. The
research of Dr. Divya Khurana and Dr. Debmalya Sain is sponsored by
Dr. D. S. Kothari Postdoctoral Fellowship under the mentorship of
Professor Gadadhar Misra.

\bibliographystyle{amsplain}

\end{document}